\documentclass[final,3p,preprint]{elsarticle}

\usepackage{amssymb}
\usepackage{amsmath}
\usepackage{lmodern}
\usepackage{booktabs}
\usepackage{multirow}
\usepackage{mathrsfs}
\usepackage{bm}
\usepackage{ntheorem}
\usepackage{color,float}
\usepackage{comment}

\usepackage{epsfig}
\usepackage{float}
\usepackage{verbatim}
\usepackage{graphicx}
\usepackage{subfigure}
\usepackage{verbatim}
\usepackage{extarrows}
\usepackage{lscape}

\numberwithin{equation}{section}

\newtheorem{theorem}{Theorem}[section]

\newtheorem{lemma}[theorem]{Lemma}
\newtheorem*{proof}{Proof.}

\newtheorem{corollary}{Corollary}[theorem]

\begin{document}
\begin{frontmatter}
    \title{Compact difference method for Euler-Bernoulli beams and plates with nonlinear nonlocal strong damping}

\author[address2]{Tao Guo}
\ead{guotao6613@163.com}

\author[address3]{Yiqun Li}
\ead{YiqunLi24@outlook.com}

\author[address1]{Wenlin Qiu\corref{mycorrespondingauthor}}
\ead{wlqiu@sdu.edu.cn}

\cortext[mycorrespondingauthor]{Corresponding author.}

\address[address2]{School of Mathematics and Statistics, Hunan Normal University, Changsha, Hunan 410081, China}

\address[address3]{School of Mathematics and Statistics, Wuhan University, Wuhan 430072, P. R. China.}

\address[address1]{School of Mathematics, Shandong University, Jinan, Shandong 250100, China}

\begin{abstract} 
We investigate the numerical approximation to the Euler-Bernoulli (E-B) beams and plates with nonlinear nonlocal strong damping, which describes the damped mechanical behavior of beams and plates in real applications. We discretize the damping term by the composite Simpson's rule and the six-point Simpson's formula in the beam and plate problems, respectively, and then construct the fully discrete compact difference scheme for these problems. To account for the nonlinear-nonlocal term, we design several novel discrete norms to facilitate the error estimates of the damping term and the numerical scheme. The stability, convergence, and energy dissipation properties of the proposed scheme are proved, and numerical experiments are carried out to substantiate the theoretical findings.
\end{abstract}

\begin{keyword} 
   Euler-Bernoulli beams and plates, nonlinear nonlocal strong damping, compact difference, stability and convergence, error estimates, energy dissipation
\end{keyword}


\end{frontmatter}

\section{Introduction}

This work considers the following Euler-Bernoulli (E-B) equation with nonlinear-nonlocal damping \cite{Lange}
\begin{equation}\label{eq4.1}
   \begin{split}
       \partial_{t}^2 u  + \hat{q}(t)\partial_{t} u + \Delta^2 u  = f(\mathbf{x},t), \quad  (\mathbf{x},t) \in \Gamma \times (0,T]
   \end{split}
   \end{equation}
equipped  with the following initial and hinged boundary conditions
     \begin{align}\label{eq4.3}
          &u(\mathbf{x},0)=u_0(\mathbf{x}),  \quad  \partial_{t}u(\mathbf{x},0)=u_1(\mathbf{x}),  \quad  \mathbf{x} \in \Gamma,\\
          \label{eq4.4}
          &u(\mathbf{x},t) = \Delta u(\mathbf{x},t)=0,  \quad \mathbf{x}\in \partial \Gamma,\;  t\in (0,T],
     \end{align}
which describes the damped mechanical behavior of beams (1D) and plates (2D) in real applications.
Here $\Gamma \subset \mathbb R^d $  $(1 \le d \le 2)$   with the boundary $\partial\Gamma$ and the strong damping coefficient is defined as follows \cite{Cannarsa,Emm}
\begin{equation}\label{eq4.2}
     \begin{split}
          \hat{q}(t) := P\left( \int_{\Gamma} \left| \Delta u(\mathbf{x},t)\right|^2 d\mathbf{x} \right)
     \end{split}
     \end{equation}
with
$P(z) : \mathbb{R}^+ \rightarrow (0,\infty)$, and $u_0(\mathbf{x})$, $u_1(\mathbf{x})$ and $f(\mathbf{x},t)$ are given functions.
We note that the nonlocal nonlinear damping term $\hat{q}(t)\partial_t u$ simulates the friction mechanism acting on the object based on the average value of $u(\mathbf{x},t)$ \cite{Cavalcanti1}, which is crucial for modulating the amplitudes of beams and plates via damping effects and thus ensures the stability and energy dissipation efficiency of the system.
The mathematical analysis of the E-B equation has attracted extensive research activities in the literature, such as the well-posedness and solution regularity \cite{Bauchau,Cavalcanti,Cavalcanti1,Emm,Hasanov,Papanicolaou,Yang}, uniform energy decay rates \cite{Bartolomeo,Cavalcanti,Lange,Lasiecka}, Riesz basis properties and stability analysis \cite{Ammari1,Guo,Liu1,Liu2,WangJ}, and boundary stabilization and controllability \cite{Horn,Li,LiuZ,Kim,Wang1}. Numerical studies for this problem is relatively rare in the literature \cite{Ahn,Weeger,Xu,Zhao}.
To the author's best knowledge, the rigorous numerical analysis of the problem \eqref{eq4.1}-\eqref{eq4.4} is still not available in the literature. The main reason lies in the fact that the nonlinear–nonlocal nature of the damping coefficient introduces additional challenges in the construction of the discretization scheme, as well as in the stability analysis and error estimation of the corresponding numerical scheme. To compensate for this gap, we develop the fully discrete compact finite difference scheme for the problem \eqref{eq4.1}-\eqref{eq4.4}, and accordingly prove its error estimate. Specifically, the main contributions are as follows:

($\mathbf{i}$) We employ the composite Simpson's rule and the six-point Simpson's formulation to discretize the nonlinear-nonlocal damping coefficient in 1D beam problem and 2D plate problem, respectively, and then develop the fully discrete compact difference scheme \eqref{eq3.1}--\eqref{eq3.4} and \eqref{eq4.10}--\eqref{eq4.13} for the problem \eqref{eq4.1}-\eqref{eq4.4}.

($\mathbf{ii}$) To account for the nonlinear-nonlocal term in \eqref{eq4.1}, we construct some novel discrete norms in 1D and 2D problems,  e.g., \eqref{g1eq2.2} and \eqref{g1eq4.7}, to facilitate the error estimates of the damping term and the numerical scheme, cf. \eqref{eq2.29}.

($\mathbf{iii}$) We combine the assumptions on the function $P$ to derive the stability, convergence, and energy dissipation properties of fully discrete schemes \eqref{eq3.1}--\eqref{eq3.4} and \eqref{eq4.10}--\eqref{eq4.13} for 1D beam and 2D plate equations. Numerical experiments  are carried out to substantiate the theoretical analysis.


Without loss of generality, we consider the problem   \eqref{eq4.1}--\eqref{eq4.4} on the domain $\Gamma= (0, 1)^d$ for $d = 1$ or $2$. Throughout the work, we use $\mathcal{Q}$ to denote  a generic positive constant independent of the mesh sizes which may be
different in different situations, and we make the following assumptions throughout this paper:\par
($\mathbf{A1}$) There exist positive constants $p_0$ and $p_1$ such that for $0\leq z \leq  \mathcal{Q}$, $P(z)\leq p_1$, and for $z \geq 0$, $P(z)\geq p_0$; \par
($\mathbf{A2}$) $P(z)$ is continuously differentiable function with $0\leq P'(z)\leq L$ for $z\geq 0$, where $L$ is the Lipschitz constant.


The rest of this paper is organized as follows. In Section \ref{sec2}, we develop the spatial semi-discrete scheme and provide its theoretical analysis for the 1D beam problem \eqref{eq1.1}-\eqref{eq1.3}. In Section \ref{sec3}, we prove the error estimate of  the fully discrete compact finite difference scheme for the 1D
 problem \eqref{eq1.1}-\eqref{eq1.3} and prove the energy dissipation property. In Section \ref{sec4}, we construct a fully discrete compact finite difference scheme for  the 2D plate problem \eqref{eq4.1}-\eqref{eq4.4}, and prove its error estimate. Numerical experiments are carried out in Section \ref{sec5} to substantiate the theoretical findings.

\section{Beam problem: Spatial semi-discrete scheme}\label{sec2}

For $d=1$, we formulate and analyze a spatial semi-discrete compact difference scheme for the beam problem \eqref{eq1.1}-\eqref{eq1.3}.

\subsection{Construction of spatial semi-discrete scheme}
Let $\Gamma=[0,1]$. We apply variable substitution $v= \partial_x^2u$ to transform problem \eqref{eq4.1}-\eqref{eq4.4} into
\begin{equation}\label{eq1.1}
   \begin{split}
       \partial^2_t u + \hat{q}(t)\partial_t u + \partial^2_{x}v({x},t) = f(x,t), \quad  (x,t) \in [0,1] \times (0,T],
   \end{split}
   \end{equation}
which is equipped  with the following initial and hinged boundary conditions
\begin{align}\label{eq1.2}
          &u(x,0)=u_0(x),  \quad  \partial_t u(x,0)=u_1(x),  \quad  x \in [0,1],\\
          \label{eq1.3}
          &u(0,t) = u(1,t)= v(0,t)= v(1,t) =0,  \quad  t\in (0,T],
     \end{align}
and the strong damping coefficient $\hat{q}(t) = P\left( \int_{0}^{1} \left| \partial^2_xu(x,t)\right|^2 dx \right)$.

We introduce some notations for subsequent analysis. Given a mesh $x_j=jh$, $j=0,1,\cdots, 2J$ with the spatial step $h=1/(2J)$, and $J$ is a positive integer. In the further analysis, we shall denote that $U_0(t)=U_{2J}(t)=0$ for $t\in (0,T]$. Let $\bar{\varpi}=\{x_j|0\leq j\leq 2J\}$, $\varpi=\bar{\varpi}\cap\Gamma$, where  $\partial\varpi=\bar{\varpi}\cap\partial\Gamma$. Accordingly, we define the following discrete spaces of grid functions
\begin{equation*}
 \begin{split}
    \mathcal{V}:=\{v=\{v_j\}| x_j\in\bar{\varpi} \},\;\; \text{and}\;\; \mathcal{V}^0:=\{v|v\in \mathcal{V} \;\;\text{and}\;\;v_j=0 \;\; \text{if} \;\;x_j\in\partial\varpi\}.
    \end{split}
\end{equation*}
Define some difference-quotient notations as follows
\begin{equation}\label{eq2.1}
     \begin{split}
        & u_j(t) := u(x_j, t), \quad U_j(t) \simeq u(x_j, t), \quad
        \delta_{x}U_j(t) := \frac{U_{j+1}(t)-U_j(t)}{h}, \\
        & \delta_{\Bar{x}}U_j(t):= \frac{U_{j}(t)-U_{j-1}(t)}{h}, \quad  \delta_{x}\delta_{\Bar{x}}U_j(t) := \frac{\delta_{x}U_j(t) - \delta_{\Bar{x}}U_j(t)}{h},\\
        & \mathscr{A}U_i(t):= \begin{cases}
    {\frac{U_{i+1}(t)+10U_i(t)+U_{i-1}(t)}{12} = \left(I+\frac{h^2}{12}\delta_{x}\delta_{\Bar{x}}\right)U_i(t), \; 1\leq i\leq 2J-1; }\\
    {U_i(t), \quad i=0,\;2J,}
    \end{cases}
   \end{split}
   \end{equation}
where $ I$ denotes the identical operator. Let the notations $W=(W_1, W_2, \cdots, W_{2J-1})^{\top}$ and $Q=(Q_1, Q_2, \cdots, Q_{2J-1})^{\top}$ be the real vectors. Moreover, denote the following discrete inner products and norms
\begin{equation}\label{eq2.2}
     \begin{split}
         \langle Q, W \rangle
         &= h \sum\limits_{j=1}^{2J-1} Q_j W_j, \quad \|Q\| = \sqrt{\langle Q, Q \rangle}, \quad \|Q\|_{\infty} = \max\limits_{1\leq j \leq 2J-1}\left| Q_j\right|.
     \end{split}
     \end{equation}
 To facilitate analysis, we define the following novel discrete norms
\begin{equation}\label{g1eq2.2}
     \begin{split}
         \|Q\|_{A}&:= \sqrt{2h \sum\limits_{j=1}^{J} Q_{2j-1}^2},\quad \|Q\|_{B}:= \sqrt{\frac{2}{3}\|Q\|^2+\frac{1}{3}\|Q\|_{A}^2}.
     \end{split}
\end{equation}

Let $U_j(t)$ and $V_j(t)$ be the numerical approximation to $u_j(t)$ and $v_j(t)$, respectively. Based on \eqref{eq2.1}, we thus obtain the following spatial semi-discrete compact difference scheme for the problem \eqref{eq1.1}-\eqref{eq1.3}
\begin{equation}\label{eq2.4}
   \begin{split}
        \mathscr{A}U_j''(t) + P\left( \|{V}(t)\|_{B}^2 \right) \mathscr{A}U_j'(t) + \delta_{x}\delta_{\Bar{x}}V_{j}(t) = \mathscr{A}f_j(t),
   \end{split}
   \end{equation}
   \begin{equation}\label{eq2.5}
   \begin{split}
        & \mathscr{A}V'_j(t) = \delta_{x}\delta_{\Bar{x}}U'_{j}(t), \quad 0< t \leq T,
   \end{split}
   \end{equation}
\begin{equation}\label{eq2.6}
   \begin{split}
        & U_0(t) = U_{2J}(t) = V_0(t)=V_{2J}(t)= 0, \quad 0< t \leq T,
   \end{split}
   \end{equation}
\begin{equation}\label{eq2.7}
   \begin{split}
        U_j(0) = u_0(x_j), \quad U_j'(0) = u_1(x_j), \quad j = 1,2,\cdots, 2J-1.
   \end{split}
   \end{equation}

We present some lemmas to facilitate analysis.
\begin{lemma}\cite{Hu}\label{lemma2.1}
  For any $\omega, v\in \mathcal{V}^0$, we have $\langle \mathscr{A}\omega,\delta_{x}\delta_{\Bar{x}}v\rangle=\langle \delta_{x}\delta_{\Bar{x}}\omega,\mathscr{A}v\rangle$.
\end{lemma}

\begin{lemma}\cite{Wang}\label{lemma2.2}
  For any $\omega\in \mathcal{V}^0$, we have $\frac{\sqrt{3}}{3}\|\omega\|\leq \|\mathscr{A}\omega\|\leq \|\omega\|$.
\end{lemma}

\begin{lemma}\cite{Lele,Sun}\label{lemma2.3}
  Assume that function $F(x)\in C^6[0,1]$ and $\zeta=5(1-s)^3-3(1-s)^5$, then it holds that
  \begin{equation*}
   \begin{split}
        \frac{F''(x_{i+1})+10F''(x_{i})+F''(x_{i-1})}{12}=\frac{F(x_{i+1})-2F(x_{i})+F(x_{i-1})}{h^2}&\\
        +\frac{h^4}{360}\int_0^1\left[F^{(6)}(x_{i}-s h)+F^{(6)}(x_{i}+s h)\right]\zeta(s)ds&, \;\;1\leq i\leq 2J-1.
   \end{split}
   \end{equation*}
\end{lemma}
\begin{lemma}\cite{Lopez}\label{lemma2.5}
 For any $V,W\in \mathcal{V}^0$, we have
  \begin{equation}\label{eq2.3}
     \begin{split}
         \langle {V}, \delta_{x}\delta_{\Bar{x}}({W}) \rangle
         = - h \sum\limits_{j=1}^{2J-1} (\delta_{x}V_j) (\delta_{x}W_j).
     \end{split}
\end{equation}
\end{lemma}
\begin{lemma}\label{lemma2.4}
 For any $\omega\in \mathcal{V}^0$, we have
  \begin{equation}\label{g12.14}
    \begin{split}
     \|\omega\|_{A}\leq \sqrt{2}\|\omega\|, \quad \frac{\sqrt{6}}{3}\|\omega\|\leq\|\omega\|_{B}\leq\frac{2\sqrt{3}}{3}\|\omega\|.
   \end{split}
   \end{equation}
\end{lemma}
\begin{proof}
 This proof could be carried out by the definition of \eqref{g1eq2.2}, and thus are omitted for simplicity.
\end{proof}

\subsection{Stability }We will prove the stability result for the semi-discrete difference scheme \eqref{eq2.4}-\eqref{eq2.7}.

Define ${U}(t):=[U_1(t), U_2(t), \cdots, U_{2J-1}(t)]^{\top}$ and $U'(t):=[U_1'(t), U_2'(t), \cdots, U_{2J-1}'(t)]^{\top}$, and then define ${V}(t)$  in an analogous manner. We accordingly define the norm
\begin{equation}\label{eq2.8}
   \begin{split}
        \|{f}\|_1 : = \int_0^{T} \|{f}(t)\|dt < \infty, \quad {f}(t)=[f_1(t), f_2(t), \cdots, f_{2J-1}(t)]^{\top},
   \end{split}
   \end{equation}
then the following stability result holds.

\begin{theorem} \label{theorem2.1}
 Suppose that ${f}$ satisfies \eqref{eq2.8} and $|\partial_x^4u_0(x)|\leq \mathcal{Q} $ on $\Gamma$, then it holds that
\begin{equation*}
 \begin{split}
     \frac{1}{2}\|{U}'(t)\|^2
       & + p_0\int_0^{t} \|U'(t)\|^2dt
      +  \frac{1}{2}\|V(t)\|^2  \leq \mathcal{Q}\left[  3\|{U}'(0)\|^2 + 3\|V(0)\|^2 + 18\|f\|_1^2\right], \quad 0\leq t\leq T,
 \end{split}
\end{equation*}
where ${V}(0)=[v_0(x_1), \cdots, v_0(x_{2J-1})]^{\top}$ and $v_0 =\partial^2_xu_0$.
\end{theorem}
\begin{proof} Take the inner product of \eqref{eq2.4} and \eqref{eq2.5} with $\mathscr{A}U'(t)$ and $\mathscr{A}V(t)$  respectively, then we have
\begin{equation*}
   \begin{split}
        \langle \mathscr{A}{U}''(t), \mathscr{A}U'(t) \rangle & +  P\left( \|V(t)\|_B^2 \right) \langle \mathscr{A}{U}'(t), \mathscr{A}U'(t) \rangle + \langle \delta_{x}\delta_{\Bar{x}}{V}(t), \mathscr{A}U'(t) \rangle = \langle \mathscr{A}{f}(t), \mathscr{A}U'(t) \rangle,
   \end{split}
   \end{equation*}
   \begin{equation*}
   \begin{split}
        \langle \mathscr{A}{V}'(t), \mathscr{A}V(t) \rangle =   \langle \delta_{x}\delta_{\Bar{x}}{U}'(t), \mathscr{A}V(t) \rangle.
   \end{split}
   \end{equation*}
We utilize the assumption $(\mathbf{A1})$ and Lemmas \ref{lemma2.1} and \ref{lemma2.5}, and add the above two formulas to obtain
\begin{equation}\label{eq2.9}
   \begin{split}
      \frac{1}{2}\frac{d}{dt}\|\mathscr{A}{U}'(t)\|^2
      & + p_0 \|\mathscr{A}{U}'(t)\|^2
      +  \frac{1}{2}\frac{d}{dt}\|\mathscr{A}{V}(t)\|^2  \leq  \| \mathscr{A}{f}(t)\| \|\mathscr{A}{U}'(t) \|.
   \end{split}
   \end{equation}
 We integrate \eqref{eq2.9} with respect to $t$ from 0 to $s$ to obtain
\begin{equation}\label{eq2.10}
   \begin{split}
       \frac{1}{2}&\|\mathscr{A}{U}'(s)\|^2
       + p_0\int_0^{s} \|\mathscr{A}{U}'(t)\|^2dt
      +  \frac{1}{2}\|\mathscr{A}{V}(s)\|^2 \\
      & \leq \frac{1}{2} \|\mathscr{A}{U}'(0)\|^2 + \frac{1}{2}\|\mathscr{A}{V}(0)\|^2   + \int_0^{s}\| \mathscr{A}{f}(t)\| \|\mathscr{A}{U}'(t) \|dt,
   \end{split}
   \end{equation}
and we use Lemma \ref{lemma2.2} and the following inequality
\begin{equation}\label{eq2.11}
   \begin{split}
       \int_0^{s}\| \mathscr{A}{f}(t)\| \|\mathscr{A}{U}'(t) \|dt \leq \frac{1}{12}\sup\limits_{0\leq t\leq s}\|{U}'(t)\|^2+3\left(\int_0^{s}\|f(t)\|dt\right)^2
   \end{split}
   \end{equation}
 to obtain that
\begin{equation}\label{eq2.12}
   \begin{split}
       \frac{1}{2}\|{U}'(s)\|^2
       & + p_0\int_0^{s} \|{U}'(t)\|^2dt
      + \frac{1}{2}\|{V}(s)\|^2 \\
      & \leq \frac{3}{12}\sup\limits_{0\leq t\leq s}\|{U}'(t)\|^2+9\left(\int_0^{s}\|f(t)\|dt\right)^2  + \frac{3}{2}\left(\|{U}'(0)\|^2+\|{V}(0)\|^2 \right)  \\
      & \leq \frac{1}{4}\sup\limits_{0\leq t\leq s}\|{U}'(t)\|^2+9\|f\|_1^2  + \frac{3}{2}\left(\|{U}'(0)\|^2+\|{V}(0)\|^2 \right).
   \end{split}
   \end{equation}
We choose an appropriate $\hat{t}$ such that $\|{U}'(\hat{t}) \|^2=\sup\limits_{0\leq t \leq s}\|{U}'(t) \|^2$ and conclude from equation \eqref{eq2.12} that
\begin{equation}\label{eq2.13}
   \begin{split}
       \frac{1}{4}\sup\limits_{0\leq t \leq s}\|{U}'(t) \|^2
       &
      \leq 9\|f\|_1^2  + \frac{3}{2}\left(\|{U}'(0)\|^2+\|{V}(0)\|^2 \right).
   \end{split}
   \end{equation}
By combining \eqref{eq2.12} and \eqref{eq2.13}, we arrive at
\begin{equation}\label{eq2.14}
   \begin{split}
        \frac{1}{2}\|{U}'(s)\|^2
       & + p_0\int_0^{s} \|{U}'(t)\|^2dt
      + \frac{1}{2}\|{V}(s)\|^2  \leq 18\|f\|_1^2 + 3\left(\|{U}'(0)\|^2+\|{V}(0)\|^2 \right).
   \end{split}
   \end{equation}
Therefore, for any $0<s<T$, we apply the Gr\"{o}nwall lemma to finish the proof.
\end{proof}
\subsection{Convergence analysis}
According to Lemma \ref{lemma2.3}, we first introduce the following key formulas before establishing the convergence analysis
\begin{equation}\label{eq2.15}
   \begin{array}{ll}
   \mathscr{A}\partial^2_xv(x_j,t)=\delta_{x}\delta_{\Bar{x}}v(x_j,t) + \mathcal{R}_1(x_j,t),\\
    \mathscr{A}\partial^2_xu(x_j,t)=\delta_{x}\delta_{\Bar{x}}u(x_j,t) + \mathcal{R}_2(x_j,t),
   \end{array}
   \end{equation}
where $\left|\mathcal{R}_1(x_j,t)\right| \leq \mathcal{Q}h^4$ and $\left|\mathcal{R}_2(x_j,t)\right| \leq \mathcal{Q}h^4$ for $1\leq j\leq 2J-1$ and  $t \in [0, T]$. Then, we consider \eqref{eq1.1}-\eqref{eq1.3} at the point $x_j$ and apply the compact operator $\mathscr{A}$ to get
\begin{equation}\label{eq2.16}
   \begin{split}
        \mathscr{A}u_j''(t) & + P\left( \int_0^1|v(x,t)|^2dx \right)\mathscr{A}u_j'(t) + \delta_{x}\delta_{\Bar{x}}v_j(t) = \mathscr{A}f_j(t) + \mathcal{R}_1(x_j, t),
   \end{split}
   \end{equation}
\begin{equation}\label{eq2.17}
   \begin{split}
        & \mathscr{A}v'_j(t) = \delta_{x}\delta_{\Bar{x}}u'_{j}(t)+\mathcal{R}_2(x_j,t), \quad 0< t \leq T,
   \end{split}
   \end{equation}
   \begin{equation}\label{eq2.18}
   \begin{split}
        & u_0(t) = u_{2J}(t) = v_0(t)=v_{2J}(t)= 0, \quad 0< t \leq T,
   \end{split}
   \end{equation}
\begin{equation}\label{eq2.19}
   \begin{split}
        u_j(0) = u_0(x_j), \quad u_j'(0) = u_1(x_j), \quad j = 1,2,\cdots, 2J-1,
   \end{split}
   \end{equation}
in which $\mathcal{R}_1(x_j, t)$ and $\mathcal{R}_2(x_j, t)$ are defined in \eqref{eq2.15}. Denote ${\xi}(t)={u}(t)-{U}(t)$ and ${\eta}(t)={v}(t)-{V}(t)$, where ${u}(t)=[u_1(t), u_2(t), \cdots, u_{2J-1}(t)]^{\top}$ and ${v}(t)=[v_1(t), v_2(t), \cdots, v_{2J-1}(t)]^{\top}$. By subtracting \eqref{eq2.4}-\eqref{eq2.7} from \eqref{eq2.16}-\eqref{eq2.19}, we obtain the error equations as follows
\begin{equation}\label{eq2.20}
   \begin{split}
        & \mathscr{A}\xi_j''(t) + P\left( \|V(t)\|_B^2 \right) \mathscr{A}\xi_j'(t) + \delta_{x}\delta_{\Bar{x}}\eta_{j}(t)  \\
        &  = R_1(x_j, t)  + \left[  P\left( \|V(t)\|_B^2 \right) -  P\left(
        \int_0^1|v(x,t)|^2 dx \right)  \right]\mathscr{A}u_j'(t),
    \end{split}
   \end{equation}
   \begin{equation}\label{eq2.21}
   \begin{split}
        & \mathscr{A}\eta'_j(t) = \delta_{x}\delta_{\Bar{x}}\xi'_{j}(t)+\mathcal{R}_2(x_j,t), \quad 0< t \leq T,
   \end{split}
   \end{equation}
\begin{equation}\label{eq2.22}
   \begin{split}
        & \xi_0(t) = \xi_{2J}(t) = 0, \quad
        \eta_{0}(t) = \eta_{2J}(t) = 0, \quad 0< t \leq T, \\
        & \xi_j(0) = 0, \quad \xi_j'(0) = 0, \quad j=1,2,\cdots, 2J-1.
   \end{split}
   \end{equation}
Next, we further present the following theorem.
\begin{theorem} \label{theorem2.2}
 Suppose that $|\partial_x^6 u(x,t)|,\;|\partial_t u(x,t)|\leq \mathcal{Q}$ on $\Gamma\times[0,T]$,
then it follows that
\begin{equation*}
 \begin{split}
    \sqrt{\|{u}'(t)-{U}'(t)\|^2
       + p_0\int_0^{t} \|{u}'(t)-{U}'(t)\|^2dt
      +  \|v(t)-{V}(t)\|^2}  \leq \mathcal{Q}h^4,\quad 0\leq t\leq T,
 \end{split}
\end{equation*}
\end{theorem}
 where $u'(t):=[u_1'(t), u_2'(t), \cdots, u_{2J-1}'(t)]^{\top}$.
\begin{proof}
   We take the inner product of \eqref{eq2.20} and \eqref{eq2.21} with $\mathscr{A}{\xi}'(t)$ and $\mathscr{A}{\eta}(t)$, respectively, to obtain that
\begin{equation}\label{eq2.23}
   \begin{split}
       & \langle \mathscr{A}{\xi}''(t), \mathscr{A}\xi'(t) \rangle  +  P\left( \|V(t)\|_B^2 \right) \langle \mathscr{A}{\xi}'(t), \mathscr{A}\xi'(t) \rangle + \langle \delta_{x}\delta_{\Bar{x}}{\eta}(t), \mathscr{A}\xi'(t) \rangle \\
       & = \langle \mathcal{R}_1(t), \mathscr{A}\xi'(t) \rangle + \left\langle\left[  P\left( \|V(t)\|_B^2 \right) -  P\left(
        \int_0^1|v(x,t)|^2 dx \right)  \right]\mathscr{A}u_j'(t), \mathscr{A}\xi'(t)\right\rangle,
    \end{split}
   \end{equation}
   \begin{equation}\label{eq2.24}
   \begin{split}
         \langle \mathscr{A}{\eta}'(t), \mathscr{A}\eta(t) \rangle =   \langle \delta_{x}\delta_{\Bar{x}}{\xi}'(t), \mathscr{A}\eta(t) \rangle + \langle \mathcal{R}_2(t), \mathscr{A}\eta(t) \rangle.
   \end{split}
   \end{equation}
We add the above two equalities and use Lemma \ref{lemma2.1} to obtain
\begin{equation}\label{eq2.25}
   \begin{split}
        & \langle \mathscr{A}{\xi}''(t), \mathscr{A}\xi'(t) \rangle  +  P\left( \|V(t)\|_B^2 \right) \langle \mathscr{A}{\xi}'(t), \mathscr{A}\xi'(t) \rangle + \langle \mathscr{A}{\eta}'(t), \mathscr{A}\eta(t) \rangle\\
       & = \left\langle\left[  P\left( \|V(t)\|_B^2 \right) -  P\left(
        \int_0^1|v(x,t)|^2 dx \right)  \right]\mathscr{A}u_j'(t), \mathscr{A}\xi'(t) \right\rangle +\langle \mathcal{R}_1(t), \mathscr{A}\xi'(t) \rangle+\langle \mathcal{R}_2(t), \mathscr{A}\eta(t) \rangle.
   \end{split}
   \end{equation}
Afterwards, we analyze the first term of the right-hand side of \eqref{eq2.25}. First, we restate
\begin{equation*}
    \begin{split}
          \int_0^1 |v(x,t)|^2dx = \sum\limits_{j=1}^{J} \int_{x_{2j-1}}^{x_{2j}}|v(x,t)|^2dx, \quad t\geq 0.
    \end{split}
\end{equation*}
We follow the well-known composite Simpson's formula and the assumptions of the theorem to derive
\begin{equation*}
    \begin{split}
          \left| \sum\limits_{j=1}^J\int_{x_{2j-1}}^{x_{2j}}|v(x,t)|^2dx -h\sum\limits_{j=1}^J  \frac{(v(x_{2j-2},t))^2+4(v(x_{2j-1},t))^2+(v(x_{2j},t))^2}{3}  \right|
          \leq \mathcal{Q}h^4.
    \end{split}
\end{equation*}
According to \eqref{eq2.18}, it holds that
\begin{equation}\label{eq2.26}
    \begin{split}
          & h\sum\limits_{j=1}^J  \frac{(v(x_{2j-2},t))^2+4(v(x_{2j-1},t))^2+(v(x_{2j},t))^2}{3}\\
          & =\frac{h}{3}\left[2\sum\limits_{j=1}^J(v(x_{2j-2},t))^2+4\sum\limits_{j=1}^J(v(x_{2j-1},t))^2\right]\\
          & =\frac{2h}{3}\left[\sum\limits_{j=1}^{2J}(v(x_{j},t))^2+\sum\limits_{j=1}^J(v(x_{2j-1},t))^2\right]\\
          & =\frac{2}{3}\left(h\sum\limits_{j=1}^{2J}(v(x_{j},t))^2\right)+\frac{1}{3}\left(2h\sum\limits_{j=1}^{J}(v(x_{2j-1},t))^2\right)\\
          & = \frac{2}{3}\|v(t)\|^2+\frac{1}{3}\|v(t)\|^2_A=\|v(t)\|_B^2.
    \end{split}
\end{equation}
 By the composite Simpson's formula and the above inequality, we have
\begin{equation}\label{eq2.27}
   \begin{split}
        \left| \int_0^1 |v(x,t)|^2dx - \|v(t)\|_B^2 \right| \leq \mathcal{Q}h^4, \quad t\geq 0.
   \end{split}
   \end{equation}
According to the assumption ($\mathbf{A2}$) and \eqref{eq2.27}, one yields
\begin{equation}\label{eq2.28}
   \begin{split}
        \left| P\left(\int_0^1 |v(x,t)|^2dx \right) - P\left(\|v(t)\|_B^2 \right) \right|  \leq \mathcal{Q}L h^4, \quad t\geq 0.
   \end{split}
   \end{equation}
Then, utilizing Lemma \ref{lemma2.4}, the triangle inequality and the assumption ($\mathbf{A2}$) to arrive at
\begin{equation}\label{eq2.29}
   \begin{split}
        \left| P\left(\|{v}(t)\|_B^2 \right) - P\left(\|V(t)\|_B^2 \right) \right|  & \leq \mathcal{Q}L\left(\|{v}(t)\|_B+ \|V(t)\|_B\right) \|{v}(t)-V(t)\| \\
        & \leq   \mathcal{Q}\|{v}(t)-V(t)\| \leq  \mathcal{Q}\|\eta(t)\|.
   \end{split}
   \end{equation}
Combining \eqref{eq2.28}-\eqref{eq2.29}, it holds that
\begin{equation}\label{eq2.30}
   \begin{split}
        \left| P\left( \|V(t)\|_B^2 \right) -  P\left(
        \int_0^1|v(x,t)|^2 dx \right)  \right|  & \leq    \mathcal{Q}\left( h^4+ \|\eta(t)\| \right).
   \end{split}
   \end{equation}
Next, we substitute \eqref{eq2.30} into \eqref{eq2.25}, use the Cauchy-Schwarz inequality and $| {u}'(t)|\leq  \mathcal{Q}$, then 
\begin{equation}\label{eq2.31}
   \begin{split}
      &\frac{1}{2}\frac{d}{dt}\|\mathscr{A}{\xi}'(t)\|^2
       + p_0 \|\mathscr{A}{\xi}'(t)\|^2
      +  \frac{1}{2}\frac{d}{dt}\|\mathscr{A}{\eta}(t)\|^2 \\
        &  \leq  \mathcal{Q}\sum\limits_{m=1}^{2} \| {R}_m(t)\|^2 +  \mathcal{Q}\left( h^8+ \|{\eta}(t)\|^2 \right) + \frac{p_0}{2}\|\mathscr{A} {\xi}'(t) \|^2+\|\mathscr{A}\eta(t)\|^2.
   \end{split}
   \end{equation}
We integrate \eqref{eq2.31} regarding $t$ from 0 to $s$, apply Lemma \ref{lemma2.2} and note ${\xi}'(0)={\eta}(0)=0$ to get
\begin{equation}\label{eq2.32}
   \begin{split}
      &\frac{1}{2}\|{\xi}'(s)\|^2
       + \frac{p_0}{2} \int_0^s\|{\xi}'(t)\|^2dt
      +  \frac{1}{2}\|{\eta}(s)\|^2  \leq  \mathcal{Q}\int_0^s\left[\sum\limits_{m=1}^{2} \| {R}_m(t)\|^2+h^8\right]dt + \mathcal{Q}\int_0^s\|{\eta}(t)\|^2dt, \quad s \leq T.
   \end{split}
   \end{equation}
Then, we use \eqref{eq2.15} to yield
\begin{equation}\label{eq2.33}
   \begin{split}
      \sum\limits_{m=1}^{2} \| {R}_m(t)\|^2 \leq  \mathcal{Q}(T) h^8, \quad t\geq 0.
   \end{split}
   \end{equation}
By inserting \eqref{eq2.33} into \eqref{eq2.32}, an application of the Gr\"{o}nwall's lemma with $T<\infty$ completes this proof.
\end{proof}

\section{Beam problem: Fully discrete scheme}\label{sec3}
In this section, we shall construct a fully discrete difference scheme and deduce some theoretical results for the problem \eqref{eq1.1}-\eqref{eq1.3}.
\subsection{Construction of fully discrete scheme}
Denote $u_j^n:=u(x_j,t_n)$ and $v_j^n:=v(x_j,t_n)$ for $0\leq j \leq 2J$ and $0\leq n \leq N+1$. Define the uniform temporal partition $t_n=n\tau$ for $0\leq n \leq N+1$.
In addition, we define some difference quotient notations as follows
\begin{equation*}
 \begin{split}
    \delta_tU^n_j&:=\frac{U^n_j-U^{n-1}_j}{\tau}, \quad  \delta_t^{(2)} U^n_j:=\delta_t(\delta_t U^{n+1}_j), \\
    \bar{\delta_t}U^n_j&:=\frac{U^{n+1}_j-U^{n-1}_j}{2\tau},\quad \widetilde{U}^n_j:=\frac{U^{n+1}_j+U^{n-1}_j}{2}, \quad n\geq 1.
    \end{split}
\end{equation*}
 Let $U_j^n$ and $V_j^n$ be numerical approximations of $u_j^n$ and $v_j^n$, respectively. Based on the above notations and the semi-discrete compact difference scheme \eqref{eq2.4}-\eqref{eq2.7}, we derive the following fully discrete compact difference scheme
\begin{equation}\label{eq3.1}
   \begin{split}
        \mathscr{A}\delta_t^{(2)} U_j^n + P\left( \|V^n\|_B^2 \right) \mathscr{A} \bar{\delta_t} U_j^n +  \delta_{x}\delta_{\Bar{x}}\widetilde{V}_j^n= \mathscr{A}f_j^n, \quad n\geq 1,
   \end{split}
   \end{equation}
   \begin{equation}\label{eq3.2}
   \begin{split}
        \mathscr{A}\bar{\delta_t}V_j^n = \delta_{x}\delta_{\Bar{x}}\bar{\delta_t}U_j^n, \quad n\geq 1,
   \end{split}
   \end{equation}
\begin{equation}\label{eq3.3}
   \begin{split}
        & U_0^n = U_{2J}^n = 0, \quad V_{0}^n = V_{2J}^n=0, \quad  1\leq n \leq N+1,
   \end{split}
   \end{equation}
\begin{equation}\label{eq3.4}
   \begin{split}
        U_j^0 = u_0(x_j), \quad \delta_tU_j^1 = u_1(x_j)+\frac{\tau}{2}u_2, \quad V^0=\Delta U^0, \quad V^1=\Delta U^1,
   \end{split}
   \end{equation}
where $j = 1,2,\cdots, 2J-1$. By \eqref{eq1.1} and \eqref{eq3.4}, we obtain the following relation
\begin{equation}\label{eq3.5}
   \begin{split}
       U_j^1 &= U_j^0 + \tau u_1(x_j)+ \frac{\tau^2}{2}u_2 =  u_0(x_j) + \tau u_1(x_j)+ \frac{\tau^2}{2}u_2, \\
       u_2 &= \partial^2_tu(0)=-\hat{q}(0)u_1-\Delta^2u_0+f^0, \quad j = 1,2,\cdots, 2J-1.
   \end{split}
\end{equation}
Throughout this section, we assume that the following regularity assumptions hold
\begin{equation*}
 \begin{split}
   (\mathbf{A3}) \;|\partial_t^4u(x,t)|\leq \mathcal{Q}, \quad |\partial_t^3\partial_x^2u(x,t)|\leq \mathcal{Q},\quad|\partial_x^6u(x,t)|\leq \mathcal{Q},\quad |\partial_x^4 u_1(x)|\leq \mathcal{Q}, \quad (x,t) \in \Gamma\times [0,T].
 \end{split}
\end{equation*}
\subsection{Stability and energy dissipation}
We prove the stability result of the fully discrete compact difference scheme \eqref{eq3.1}-\eqref{eq3.4}. We still denote ${U}^n:=[U_1^n, U_2^n, \cdots, U_{2J-1}^n]^{\top}$ and ${V}^n:=[V_1^n, V_2^n, \cdots, V_{2J-1}^n]^{\top}$ to be the solution of \eqref{eq3.1}-\eqref{eq3.4} when no confusion occurs.

\begin{theorem}\label{theorem3.1}  Suppose that \eqref{eq2.8} and the assumption ($\mathbf{A3}$) hold, then we have
  \begin{equation*}
   \begin{split}
          \|U^N\|_C\leq \|U^0\|_C+2\tau\sum_{j=1}^N\|f^j\|,
   \end{split}
   \end{equation*}
where $\|U^n\|_C:=\sqrt{\|\mathscr{A}\delta_t {U}^{n+1}\|^2+\frac{1}{2}\left(\|\mathscr{A}V^{n+1}\|^2+\|\mathscr{A}V^{n}\|^2\right)}$
\end{theorem}
\begin{proof} We take the inner product of \eqref{eq3.1} with $\mathscr{A}\bar{\delta_t} {U}^n$ and \eqref{eq3.2} with $\mathscr{A} \widetilde{V}^n$, respectively, to get
\begin{equation}\label{eq3.6}
   \begin{split}
         \langle \mathscr{A}\delta_t^{(2)}{U}^n, \mathscr{A}\bar{\delta_t} {U}^n \rangle  +  P\left( \|V^n\|_B^2 \right) \|\mathscr{A}\bar{\delta_t} {U}^n\|^2 + \langle \delta_x\delta_{\Bar{x}}\widetilde{V}^n, \mathscr{A}\bar{\delta_t} {U}^n \rangle
         = \langle \mathscr{A}{f}^n, \mathscr{A}\bar{\delta_t} {U}^n \rangle,
   \end{split}
   \end{equation}
   \begin{equation}\label{eq3.7}
   \begin{split}
         \langle \mathscr{A}\bar{\delta_t}{V}^n, \mathscr{A} \widetilde{V}^n \rangle
         = \langle \delta_x\delta_{\Bar{x}}\bar{\delta_t}U^n, \mathscr{A} \widetilde{V}^n \rangle,\quad n \geq 1.
   \end{split}
   \end{equation}
Then, adding \eqref{eq3.6} and \eqref{eq3.7} together and using Lemma \ref{lemma2.1}, we yield that
\begin{equation}\label{eq3.8}
   \begin{split}
         \langle \mathscr{A}\delta_t^{(2)}{U}^n, \mathscr{A}\bar{\delta_t} {U}^n \rangle  +  P\left( \|V^n\|_B^2 \right) \|\mathscr{A}\bar{\delta_t} {U}^n\|^2 + \langle \mathscr{A}\bar{\delta_t}{V}^n, \mathscr{A} \widetilde{V}^n \rangle
         = \langle \mathscr{A}{f}^n, \mathscr{A}\bar{\delta_t} {U}^n \rangle.
   \end{split}
   \end{equation}
First,  we notice that
\begin{equation}\label{eq3.9}
   \begin{split}
        & \langle \mathscr{A}\delta_t^{(2)} {U}^n, \mathscr{A}\bar{\delta_t}{U}^n \rangle
        = \frac{1}{2\tau} \left[ \|\mathscr{A}\delta_t {U}^{n+1}\|^2 -  \|\mathscr{A}\delta_t {U}^{n}\|^2 \right], \\
        & \langle  \mathscr{A}\bar{\delta_t}{V}^n, \mathscr{A} \widetilde{V}^n \rangle
        = \frac{1}{4\tau}\left[ \| \mathscr{A}{V}^{n+1}\|^2 - \|\mathscr{A}{V}^{n-1}\|^2 \right].
   \end{split}
   \end{equation}
Then we substitute \eqref{eq3.9} into \eqref{eq3.8} to arrive at
\begin{equation}\label{eq3.10}
   \begin{split}
        &\frac{1}{2\tau} \left[ \|\mathscr{A}\delta_t {U}^{n+1}\|^2 -  \|\mathscr{A}\delta_t {U}^{n}\|^2 \right] + p_0\|\mathscr{A}\bar{\delta_t} {U}^n\|^2 \\
        &+\frac{1}{4\tau}\left[ \| \mathscr{A}{V}^{n+1}\|^2 +\| \mathscr{A}{V}^{n}\|^2 -\left(\| \mathscr{A}{V}^{n}\|^2 + \|\mathscr{A}{V}^{n-1}\|^2 \right)\right] \\
        &\leq \|\mathscr{A}{f}^n\| \|\mathscr{A}\bar{\delta_t} {U}^n\|,
   \end{split}
   \end{equation}
where the assumption ($\mathbf{A1}$) is applied. Summing for \eqref{eq3.10} regarding $n$ from $1$ to $N$, multiplying the result inequality by $2\tau $ and using Lemma \ref{lemma2.2}, we obtain that
\begin{equation}\label{eq3.11}
   \begin{split}
        &\qquad \qquad\qquad\|\mathscr{A}\delta_t {U}^{N+1}\|^2 - \|\mathscr{A}\delta_t {U}^{1}\|^2 + 2\tau p_0\sum_{n=1}^N\|\mathscr{A}\bar{\delta_t} {U}^n\|^2 \\
        & +\frac{1}{2}\left[ \|\mathscr{A} {V}^{N+1}\|^2 +\| \mathscr{A}{V}^{N}\|^2 -\left(\|\mathscr{A}{V}^{1}\|^2 + \|\mathscr{A}{V}^{0}\|^2 \right)\right]\leq 2\tau\sum_{n=1}^N\|{f}^n\| \|\mathscr{A}\bar{\delta_t} {U}^n\|.
   \end{split}
   \end{equation}
Therefore, \eqref{eq3.11} is transformed into
\begin{equation}\label{eq3.13}
   \begin{split}
         \|U^N\|_C^2+2\tau p_0\sum_{n=1}^N\|\bar{\delta_t} {U}^n\|^2\leq \|U^0\|_C^2+2\tau\sum_{n=1}^N\|{f}^n\| \|\mathscr{A}\bar{\delta_t} {U}^n\|.
   \end{split}
   \end{equation}
According to the definition of $\|\cdot\|_C$, which is noted that
\begin{equation}\label{eq3.14}
   \begin{split}
        &\|\mathscr{A}\bar{\delta_t} {U}^n\|=\left\|\frac{\mathscr{A}U^{n+1}-\mathscr{A}U^{n}+\mathscr{A}U^{n}-\mathscr{A}U^{n-1}}{2\tau}\right\|\\
        &\leq \frac{\|\delta_t \mathscr{A}U^{n+1}\|+\|\delta_t \mathscr{A}U^{n}\|}{2}\leq \frac{\| U^{n}\|_C+\|U^{n-1}\|_C}{2}.
   \end{split}
   \end{equation}
Then, we substitute \eqref{eq3.14} into \eqref{eq3.13} to yield
\begin{equation}\label{eq3.15}
   \begin{split}
       \|U^N\|_C^2 \leq \|U^0\|_C^2+2\tau\sum_{n=1}^N\|{f}^n\| \frac{\| U^{n}\|_C+\|U^{n-1}\|_C}{2}.
   \end{split}
   \end{equation}
By taking a suitable $m$ such that $\|U^m\|_C=\max\limits_{0\leq n \leq N}\|U^n\|_C$, we have
\begin{equation*}
   \begin{split}
              \|U^N\|_C \leq \|U^m\|_C \leq  \|U^0\|_C+2\tau\sum_{n=1}^N\|{f}^n\|.
   \end{split}
   \end{equation*}
This completes the proof.
\end{proof}

Based on Theorem \ref{theorem3.1}, we prove the energy dissipative property of the fully discrete compact difference scheme in the following theorem.

\begin{corollary}\label{coro1}
    The fully discrete compact difference scheme \eqref{eq3.1}-\eqref{eq3.4} satisfies the following  energy dissipative property
    \begin{equation*}
        0\leq\hat{\mathbf{E}}^n \leq \hat{\mathbf{E}}^0, \quad 1\leq n \leq N,
        \end{equation*}
where
 \begin{equation}\label{energy1d}
     \hat{\mathbf{E}}^n := \sqrt{\|\mathscr{A}\delta_t {U}^{n+1}\|^2+\frac{1}{2}\left(\|\mathscr{A}V^{n+1}\|^2+\|\mathscr{A}V^{n}\|^2\right)}.
    \end{equation}
\end{corollary}
\begin{proof}
We set $f\equiv 0$. The proof could be carried out by following Theorem \ref{theorem3.1}, and thus is omitted for simplicity.
\end{proof}
\subsection{Convergence analysis}
We prove the convergence result of the fully discrete compact difference scheme \eqref{eq3.1}-\eqref{eq3.4}.
We consider \eqref{eq2.16} at the point $t=t_n$ for $1\leq n\leq N$
\begin{equation}\label{eq3.16}
   \begin{split}
        \mathscr{A}\delta_t^{(2)} u_j^n &+ P\left( \int_0^1 |v(x,t_n)|^2dx \right) \mathscr{A}\bar{\delta_t} u_j^n + \delta_x\delta_{\bar{x}}\widetilde{v}_j^n = \mathscr{A}{f}_j^n + (R_1)_j^n + \sum\limits_{m=3}^{5}(R_m)_j^n,
   \end{split}
   \end{equation}
   \begin{equation}\label{eq3.17}
   \begin{split}
        \mathscr{A}\bar{\delta_t}v_j^n = \delta_x\delta_{\bar{x}}\bar{\delta_t}u_j^n + (R_2)_j^n,
   \end{split}
   \end{equation}
\begin{equation}\label{eq3.18}
   \begin{split}
        & u_0^n = u_{2J}^n = 0, \quad v_0^n = v_{2J}^n = 0,
   \end{split}
   \end{equation}
\begin{equation}\label{eq3.19}
   \begin{split}
        u_j^0 = u_0(x_j), \quad (\partial_tu)_j^0 = u_1(x_j), \quad j = 1,2,\cdots, 2J-1,
   \end{split}
   \end{equation}
in which the truncation errors
\begin{equation}\label{eq3.20}
   \begin{split}
        & (R_1)_j^n = \mathscr{A}{(\partial^2_xv_j^n)}-\delta_{x}\delta_{\Bar{x}}v_j^n,\\
        & (R_2)_j^n = \mathcal{R}_2(x_j,t_n) + \mathscr{A}\left(\bar{\delta_t} v_j^n - v_j'(t_n)\right)+\delta_x\delta_{\bar{x}}\left(\bar{\delta_t} u_j^n - u_j'(t_n)\right),  \\
        & (R_3)_j^n = \mathscr{A}\left(\delta_t^{(2)} u_j^n - u_j''(t_n)\right), \quad  n \geq 1, \\
        & (R_4)_j^n = P\left( \int_0^1 |v(x,t_n)|^2dx \right) \mathscr{A}\left[\bar{\delta_t} u_j^n - u_j'(t_n) \right], \\
        &  (R_5)_j^n = \delta_x\delta_{\bar{x}}\left( \widetilde{v}_j^n-v(x_j, t_n)\right).
   \end{split}
   \end{equation}
Then, define $\xi_j^n := u_j(t_n)-U_j^n = u_j^n - U_j^n$ and $ \eta_j^n:=v_j^n-V_j^n$ for $0\leq n \leq N+1$, ${\xi}^n=[\xi_1^n,\xi_2^n,\cdots,\xi_{2J-1}^n]^{\top}$ and ${\eta}^n=[\eta_1^n,\eta_2^n,\cdots,\eta_{2J-1}^n]^{\top}$. We subtract \eqref{eq3.1}-\eqref{eq3.4} from \eqref{eq3.16}-\eqref{eq3.19}, to obtain the following error equations
\begin{equation}\label{eq3.21}
   \begin{split}
        \mathscr{A}\delta_t^{(2)}\xi_j^n &+ P\left( \|{V}^n\|_B^2 \right) \mathscr{A}\bar{\delta_t} \xi_j^n  + \delta_{x}\delta_{\Bar{x}}\widetilde{\eta}_j^n
         =  (R_1)_j^n \\
         &+\left[  P\left( \|V^n\|_B^2 \right) -  P\left(
        \int_0^1|v(x,t_n)|^2 dx \right)  \right] \mathscr{A}\bar{\delta_t}u_j^n + \sum\limits_{m=3}^{5}(R_m)_j^n,
   \end{split}
   \end{equation}
   \begin{equation}\label{eq3.22}
   \begin{split}
        \mathscr{A}\bar{\delta_t} \eta_j^n
         =  \delta_{x}\delta_{\Bar{x}}\bar{\delta_t}\xi_j^n+(R_2)_j^n ,\quad n\geq 1,
   \end{split}
   \end{equation}
\begin{equation}\label{eq3.23}
   \begin{split}
        & \xi_0^n = \xi_{2J}^n = 0, \quad
        \eta_0^n = \eta_{2J}^n = 0, \quad  1\leq n \leq N+1,
   \end{split}
   \end{equation}
\begin{equation}\label{eq3.24}
   \begin{split}
        \xi_j^0 = 0, \quad \delta_t \xi_j^1 = \delta_tu_j^1 - \delta_t U^1_j = \frac{1}{2\tau} \int_{0}^{t_1} \partial^3_su(x_j,s)(s-t_1)^2ds, \quad 1\leq j \leq 2J-1.
   \end{split}
   \end{equation}
Then, the following convergence result holds.
\begin{theorem} \label{theorem3.2}
 Suppose that the assumption ($\mathbf{A3}$) holds,
then we have
\begin{equation}\label{qqq1}
 \begin{split}
   &\sqrt{\|\delta_t({u}^{n+1}-{U}^{n+1})\|^2
       + \|{v}^{n+1} -{V}^{n+1} \|^2 + \|{v}^{n} -{V}^{n} \|^2}  \leq \mathcal{Q}(T) \left( \tau^2 + h^4\right), \quad 1\leq n \leq N.
 \end{split}
\end{equation}
\end{theorem}
\begin{proof} First, we make the inner product of \eqref{eq3.21} and \eqref{eq3.22} with $\mathscr{A}\bar{\delta_t}{\xi}^n$ and $\mathscr{A}\widetilde{\eta}^n$, respectively, then add the two resulting equation and apply the assumption ($\mathbf{A1}$) to obtain
\begin{equation}\label{eq3.26}
   \begin{split}
       & \langle\mathscr{A}\delta_t^{(2)}\xi^n,  \mathscr{A}\bar{\delta_t}{\xi}^n\rangle + p_0 \|\mathscr{A}\bar{\delta_t}{\xi}^n\|^2  + \langle\mathscr{A}\bar{\delta_t}\eta^n,  \mathscr{A}\widetilde{\eta}^n\rangle
         \leq \langle({R}_1)^n, \mathscr{A}\bar{\delta_t}{\xi}^n\rangle + \langle({R}_2)^n, \mathscr{A}\widetilde{\eta}^n\rangle\\
       & + \sum\limits_{m=3}^{5}\langle({R}_m)^n, \mathscr{A}\bar{\delta_t}{\xi}^n\rangle  + \left[  P\left( \|{V}^n\|_B^2 \right) -  P\left(
        \int_0^1|v(x,t_n)|^2 dx \right)  \right] \langle\mathscr{A}\bar{\delta_t}{u}^n, \mathscr{A}\bar{\delta_t}{\xi}^n\rangle.
   \end{split}
   \end{equation}
Notice that
\begin{equation}\label{eq3.27}
   \begin{split}
        & \langle \mathscr{A}\delta_t^{(2)} {\xi}^n, \mathscr{A}\bar{\delta_t}{\xi}^n \rangle
        = \frac{1}{2\tau} \left[ \|\mathscr{A}\delta_t {\xi}^{n+1}\|^2 -  \|\mathscr{A}\delta_t {\xi}^{n}\|^2 \right], \\
        & \langle \mathscr{A}\bar{\delta_t}{\eta}^n, \mathscr{A} \widetilde{\eta}^n \rangle
        = \frac{1}{4\tau}\left[ \| \mathscr{A}{\eta}^{n+1}\|^2 - \|\mathscr{A}{\eta}^{n-1}\|^2 \right].
   \end{split}
   \end{equation}
We apply \eqref{eq2.27}-\eqref{eq2.30}, \eqref{eq3.27} and the Cauchy-Schwarz inequality to yield that
\begin{equation}\label{eq3.28}
   \begin{split}
       & \frac{1}{2\tau}\left[\|\mathscr{A}\delta_t\xi^{n+1}\|^2-\|\mathscr{A}\delta_t\xi^{n}\|^2\right] + p_0 \|\mathscr{A}\bar{\delta_t}{\xi}^n\|^2 \\
       & + \frac{1}{4\tau}\left[\|\mathscr{A}\eta^{n+1}\|^2+\|\mathscr{A}\eta^{n}\|^2-\left(\|\mathscr{A}\eta^{n}\|^2+\|\mathscr{A}\eta^{n-1}\|^2\right) \right]\\
       & \leq \sum\limits_{m=3}^{5}\|({R}_m)^n\|\|\mathscr{A}\bar{\delta_t}{\xi}^n\| + \|({R}_1)^n\|\|\mathscr{A}\bar{\delta_t}{\xi}^n\|\\
       & + \|({R}_2)^n\| \|\mathscr{A}\widetilde{\eta}^n\|+ \mathcal{Q}\left(h^4+\|\eta^n\|\right)\|\mathscr{A}\bar{\delta_t}\xi^n\|,
   \end{split}
   \end{equation}
where
\begin{equation*}
    \begin{split}
          |\bar{\delta_t}u_j^n|=\left|\partial_tu(x_j,t_n)+\frac{1}{2\tau}\left(\int_{t_n}^{t_{n+1}}(t_n-s)\partial_s^2u(x_j,s)ds-\int_{t_{n-1}}^{t_{n}}(s-t_{n-1})\partial_s^2u(x_j,s)ds\right)\right|
          \leq \mathcal{Q},
    \end{split}
\end{equation*}
this implies that $\|\bar{\delta_t}u^n\|\leq \mathcal{Q}$. Next, we define
\begin{equation}\label{eq3.29}
    \begin{split}
        \|\xi^n\|_D:=\sqrt{\|\mathscr{A}\delta_t\xi^{n+1}\|^2+\frac{1}{2}\left(\|\mathscr{A}\eta^{n+1}\|^2+\|\mathscr{A}\eta^{n}\|^2\right)}.
    \end{split}
\end{equation}
Then, we sum \eqref{eq3.28} regarding $n$ from $1$ to $N$ and multiply the result inequality by $2\tau$ to obtain
\begin{equation}\label{eq3.283}
   \begin{split}
       & \|\xi^N\|_D^2 -  \|\xi^0\|_D^2 +2\tau p_0 \sum_{n=1}^N\|\mathscr{A}\bar{\delta_t}{\xi}^n\|^2 \leq 2\tau\sum_{n=1}^N\sum\limits_{m=3}^{5}\|({R}_m)^n\|\|\mathscr{A}\bar{\delta_t}{\xi}^n\| \\
       & + 2\tau\sum_{n=1}^N\|({R}_1)^n\|\|\mathscr{A}\bar{\delta_t}{\xi}^n\| + 2\tau\sum_{n=1}^N\|({R}_2)^n\| \|\mathscr{A}\widetilde{\eta}^n\|+ \mathcal{Q}\tau\sum_{n=1}^N\left(h^4+\|\eta^n\|\right)\|\mathscr{A}\bar{\delta_t}\xi^n\|.
   \end{split}
   \end{equation}
Note that
\begin{equation}\label{eq3.31}
   \begin{split}
        &\|\mathscr{A}\bar{\delta_t} {\xi}^n\|= \left\|\frac{\mathscr{A}\xi^{n+1}-\mathscr{A}\xi^{n}+\mathscr{A}\xi^{n}-\mathscr{A}\xi^{n-1}}{2\tau}\right\|\leq \frac{\|\mathscr{A}\delta_t \xi^{n+1}\|+\|\mathscr{A}\delta_t \xi^{n}\|}{2}\leq \frac{\| \xi^{n}\|_D+\|\xi^{n-1}\|_D}{2},
   \end{split}
   \end{equation}
and
\begin{equation}\label{eq3.32}
   \begin{split}
        &\|\mathscr{A}\widetilde{\eta}^n\|= \left\|\frac{\mathscr{A}\eta^{n+1}+\mathscr{A}\eta^{n-1}}{2}\right\|\leq \frac{\|\mathscr{A} \eta^{n+1}\|+\|\mathscr{A}\eta^{n-1}\|}{2}\leq \frac{\sqrt{2}}{2}\left(\| \xi^{n}\|_D+\|\xi^{n-1}\|_D\right).
   \end{split}
   \end{equation}
We now choose an appropriate $m$ such that $\|\xi^m\|_D=\max\limits_{0\leq n\leq N}\|\xi^n\|_D$ and use \eqref{eq3.31}-\eqref{eq3.32} to get
\begin{equation}\label{eq3.33}
   \begin{split}
       &\|\xi^m\|_D^2 \leq \|\xi^0\|_D\|\xi^m\|_D + \mathcal{Q}\tau\sum_{n=1}^N\sum\limits_{m=1}^{5}\|({R}_m)^n\|\|\xi^m\|_D +  \mathcal{Q}\tau\sum_{n=1}^N\left(h^4+\|\xi^n\|_D\right)\|\xi^m\|_D.
   \end{split}
   \end{equation}
 Consequently,
   \begin{equation}\label{g1eq3.33}
   \begin{split}
        \|\xi^N\|_D \leq \|\xi^m\|_D \leq \|\xi^0\|_D + \mathcal{Q}\tau\sum_{n=1}^N\sum\limits_{m=1}^{5}\|({R}_m)^n\| +  \mathcal{Q}\tau\sum_{n=1}^N\left(h^4+\|\xi^n\|_D\right).
   \end{split}
   \end{equation}
Subsequently, we shall analyse $\|\xi^0\|_D$. First, we apply $\xi^0=0$ to arrive at
\begin{equation}\label{eq3.34}
   \begin{split}
        \xi^1_j=u^1_j-U^1_j=\frac{1}{2}\int_{0}^{t_1}(s-t_1)^2\partial_s^3u(x_j,s)ds.
   \end{split}
   \end{equation}
Note that $\eta^0=0 $, we similarly can obtain
\begin{equation}\label{eq3.35}
   \begin{split}
         \eta^1_j=v^1_j-V^1_j=\frac{1}{2}\int_{0}^{t_1}(s-t_1)^2\partial_s^3v(x_j,s)ds.
   \end{split}
   \end{equation}
Consequently,
\begin{equation}\label{eq3.36}
   \begin{split}
        \|\xi^0\|_D \leq \mathcal{Q}\int_{0}^{t_1}(s-t_1)^2(\|\partial_s^3u(s)\|+\|\partial_s^3v(s)\|)ds \;\leq \mathcal{Q}\tau^2.
   \end{split}
   \end{equation}
We employ the Taylor expansion with integral remainder to get
\begin{equation*}
   \begin{split}
          &u''_j(t_n)-\delta_t^{(2)}u^n_j=\frac{-1}{6\tau^2}\left[\int_{t_n}^{t_{n+1}}(t_{n+1}-s)^3\partial_s^4u(x_j,s)ds+\int_{t_{n-1}}^{t_{n}}(s-t_{n-1})^3\partial_s^4u(x_j,s)ds\right],\\
          &n\geq 2, \quad u''_j(t_1)-\delta_t^{(2)}u^1_j=\frac{-1}{2\tau^2}\left[\int_{t_1}^{t_{2}}(t_{2}-s)^2\partial_s^3u(x_j,s)ds+\int_{0}^{t_{1}}s^2\partial_s^3u(x_j,s)ds\right],
   \end{split}
   \end{equation*}
therefore, we obtain that
\begin{equation}\label{eq3.38}
   \begin{split}
          \tau\sum_{n=1}^N\|({R}_3)^n\|\leq\tau\int_0^{2\tau}\|\partial_s^3u(s)\|ds+\tau^2\int_\tau^{t_{N+1}}\|\partial_s^3u(s)\|ds\;\leq \mathcal{Q}\tau^2. 
   \end{split}
   \end{equation}
Then for $({R}_4)_j^n $, we have
\begin{equation*}
   \begin{split}
    u'_j(t_n)-\bar{\delta_t}u^n_j=\frac{-1}{4\tau}\left[\int_{t_n}^{t_{n+1}}(t_{n+1}-s)^2\partial_s^3u(x_j,s)ds+\int_{t_{n-1}}^{t_{n}}(s-t_{n-1})^2\partial_s^3u(x_j,s)ds\right],
   \end{split}
   \end{equation*}
thus, we further understand that
\begin{equation}\label{eq3.39}
   \begin{split}
    \tau\sum_{n=1}^N\|({R}_4)^n\|\leq \frac{p_0}{2}\tau^2\int_0^{t_{N+1}}\|\partial_s^3u(s)\|ds\;\leq \mathcal{Q}\tau^2. 
   \end{split}
   \end{equation}
   Then, we use \eqref{eq3.20} and Lemma \ref{lemma2.3} to yield
\begin{equation}\label{eq3.37}
   \begin{split}
          \tau\sum_{n=1}^N\sum\limits_{m=1}^{2}\|({R}_m)^n\|\leq \mathcal{Q}\left(h^4+\tau^2\int_0^{t_{N+1}}\|\partial_s^3u(s)\|ds\right)\;\leq \mathcal{Q}(\tau^2+h^4).
   \end{split}
   \end{equation}
For the term $({R}_5)_j^n$ in \eqref{eq3.20}, we obtain that
\begin{equation*}
   \begin{split}
          v_j^n-\widetilde{v}_j^n=\frac{-1}{2}\left[\int_{t_n}^{t_{n+1}}(t_{n+1}-s)\partial_s^2v(x_j,s)ds+\int_{t_{n-1}}^{t_{n}}(s-t_{n-1})\partial_s^2v(x_j,s)ds\right],
   \end{split}
   \end{equation*}
hence,
\begin{equation}\label{eq3.40}
   \begin{split}
          \tau\sum_{n=1}^N\|({R}_5)^n\|\leq \tau^2\int_0^{t_{N+1}}\|\partial_s^2v(s)\|ds\;\leq \mathcal{Q}\tau^2. 
   \end{split}
   \end{equation}
We substitute \eqref{eq3.36}-\eqref{eq3.40} into \eqref{eq3.33} and use the discrete Gr\"onwall's lemma to yield \eqref{qqq1}. This finishes the proof.
\end{proof}

\section{Plate problem: Fully discrete scheme} \label{sec4}

This section will establish a 2D compact difference scheme and analyze its stability and convergence for problem \eqref{eq4.1}-\eqref{eq4.4}.

\subsection{Construction of fully discrete scheme}
Let $\Gamma=[0,1]^2$ and $\mathbf{x}=(x,y)$. We take two positive integers $J_1$ and $J_2$, and let spatial step sizes $h_1=\frac{1}{2J_1}$ and $h_2=\frac{1}{2J_2}$, with the spatial grid points $x_i=ih_1(0\leq i\leq 2J_1)$ and $y_j=jh_2(0\leq j\leq 2J_2)$. Suppose that $\bar{\Gamma}_h=\{(x_i,y_j)|0\leq i\leq 2J_1,\;0\leq j\leq 2J_2\}$, $\Gamma_h=\Gamma\cap \bar{\Gamma}_h$ and $\partial\Gamma_h=\bar{\Gamma}_h\cap\partial\Gamma$. In addition, we give the grid function $\Gamma_\tau=\{t_n|0\leq n\leq N+1\}$.
 For any grid function $\omega=\{\omega_{ij}^n|0\leq i\leq 2J_1,\;0\leq j\leq 2J_2,\;0\leq n\leq N+1\}$ on $\bar{\Gamma}_h\times \Gamma_\tau$, denote
\begin{equation}\label{eq4.5}
   \begin{split}
        & u_{ij}^n:=u(x_i,y_j,t_n), \quad U_{ij}^n \simeq u(x_i,y_j, t_n), \quad
        \delta_{x}U_{ij}^n := \frac{U_{i+1,j}^n-U_{ij}^n}{h_1}, \\
        & \delta_{\Bar{x}}U_{ij}^n := \frac{U_{ij}^n-U_{i-1,j}^n}{h_1}, \quad  \delta_{x}\delta_{\Bar{x}}U_{ij}^n := \frac{\delta_{x}U_{ij}^n - \delta_{\Bar{x}}U_{ij}^n}{h_1},\\
        & \mathcal{A}U_{ij}^n:= \begin{cases}
    {\frac{U_{i+1,j}^n+10U_{ij}^n+U_{i-1,j}^n}{12} = \left(I_x+\frac{h_1^2}{12}\delta_{x}\delta_{\Bar{x}}\right)U_{ij}^n, \; 1\leq i\leq 2J_1-1; }\\
    {U_{ij}^n, \quad i=0,\;2J_1,\; 0\leq j\leq 2J_2,}
    \end{cases}\\
        & \mathcal{B}U_{ij}^n:= \begin{cases}
    {\frac{U_{i,j+1}^n+10U_{ij}^n+U_{i,j-1}^n}{12} = \left(I_y+\frac{h_2^2}{12}\delta_{y}\delta_{\Bar{y}}\right)U_{ij}^n, \; 1\leq j\leq 2J_2-1; }\\
    {U_{ij}^n, \quad j=0,\;2J_2,\; 0\leq i\leq 2J_1,}
    \end{cases}
   \end{split}
   \end{equation}
where $I_x$ and $I_y$ denote the identical operator. Other symbols $\delta_{y}U_{ij}^n, \delta_{\Bar{y}}U_{ij}^n$ and $\delta_{y}\delta_{\Bar{y}}U_{ij}^n$ can be defined analogously. Furthermore, we define the following notations
\begin{equation}\label{eq4.6}
   \begin{split}
\mathcal{H}\omega_{ij}:=\mathcal{A}\mathcal{B}\omega_{ij}, \quad \Phi \omega_{ij}:=\left(\mathcal{B}\delta_{x}\delta_{\Bar{x}}+\mathcal{A}\delta_{y}\delta_{\Bar{y}}\right)\omega_{ij}.
   \end{split}
   \end{equation}
First,
 we define a grid function space $W_h:=\{\omega|\omega=\{\omega_{ij}|(x_i,y_j)\in\Gamma_h\} ,\;\omega_{ij}=0, (x_i,y_j)\in\partial\Gamma_h\}$. Then for any grid function $\omega, \nu\in W_h$, we define the following discrete inner products and norms
\begin{equation}\label{eq4.7}
   \begin{split}
   &\langle \omega, \nu \rangle :=h_1h_2\sum_{i=1}^{2J_1-1}\sum_{j=1}^{2J_2-1}\omega_{ij}\nu_{ij},\quad \|\omega\| := \sqrt{\langle \omega, \omega \rangle},\quad \|\omega\|_\infty:=\max_{1\leq i\leq 2J_1,\; 1\leq j\leq 2J_2}|\omega_{ij}|.
   \end{split}
\end{equation}
We define the following novel discrete norms to facilitate analysis
\begin{equation}\label{g1eq4.7}
   \begin{split}
    \|\omega\|_E:=\frac{2}{3}\sqrt{h_1h_2\sum_{i=1}^{J_1}\sum_{j=1}^{J_2}\left(\omega^2_{2i-2,2j-1}+\omega^2_{2i-1,2j-2}+3\omega^2_{2i-1,2j-1}\right)},\quad \|\omega\|_F:= \sqrt{\frac{4}{9}\|\omega\|^2+\|\omega\|_E^2}.
   \end{split}
\end{equation}
We then introduce the 2D six-point composite Simpson's rule \cite{Chapra, ChapraCanale}, which plays a crucial role in the approximation of the nonlinear damping term. On a small interval $[x_{2i-2}, x_{2i}]\times [y_{2j-2}, y_{2j}]$ for $1\leq i\leq J_1$ and $1\leq j\leq J_2$, the 2D six-point composite Simpson's rule is given as follows
\begin{equation}\label{eq4.9}
   \begin{split}
   &\int_{x_{2i-2}}^{x_{2i}}\int_{y_{2j-2}}^{y_{2j}}\omega(x,y)dxdy \approx\frac{h_1h_2}{36}\Big[\omega_{2i-2,2j-2}+4\omega_{2i-2,2j-1}+\omega_{2i-2,2j}+4\omega_{2i-1,2j-2}\\
   &+16\omega_{2i-1,2j-1}+4\omega_{2i-1,2j}+\omega_{2i,2j-2}+4\omega_{2i,2j-1}+\omega_{2i,2j}\Big], \quad\omega_{i,j}=\omega(x_i,y_j).
   \end{split}
   \end{equation}
 Consequently,
\begin{equation}\label{eq4.99}
   \begin{split}
   &\int_{\Gamma} \omega(x,y) dxdy \approx \sum_{i=1}^{J_1}\sum_{j=1}^{J_2}\frac{h_1h_2}{36}\Big[4\omega_{2i-2,2j-2}+8\omega_{2i-2,2j-1}+8\omega_{2i-1,2j-2}+16\omega_{2i-1,2j-1}\Big].
   \end{split}
   \end{equation}
Similar to the 1D case, we directly establish the fully discrete compact difference scheme for $1\leq i\leq 2J_1-1$, $1\leq j\leq 2J_2-1$ and $1\leq n\leq N$,
\begin{equation}\label{eq4.10}
   \begin{split}
        \mathcal{H}\delta_t^{(2)} U_{ij}^n + P\left( \|V^n\|_F^2 \right) \mathcal{H} \bar{\delta_t} U_{ij}^n +  \Phi\widetilde{V}_{ij}^n= \mathcal{H}f_{ij}^n, \quad n\geq 1,
   \end{split}
   \end{equation}
   \begin{equation}\label{eq4.11}
   \begin{split}
        \mathcal{H}\bar{\delta_t}V_{ij}^n = \Phi\bar{\delta_t}U_{ij}^n,
   \end{split}
   \end{equation}
\begin{equation}\label{eq4.12}
   \begin{split}
        & U_{0,j}^n = U_{2J_1,j}^n = U_{i,0}^n = U_{i,2J_2}^n = 0, \quad V_{0,j}^n = V_{2J_1,j}^n = V_{i,0}^n = V_{i,2J_2}^n = 0,
   \end{split}
   \end{equation}
\begin{equation}\label{eq4.13}
   \begin{split}
        U_{ij}^0 = u_0(x_i,y_j), \quad \delta_tU_{ij}^1 = u_1(x_i,y_j)+\frac{\tau}{2}u_2, \quad V^0=\Delta U^0, \quad V^1=\Delta U^1.
   \end{split}
   \end{equation}
Throughout this section, we adopt the following regularity assumptions for the solution $u(\mathbf{x},t)$
\begin{equation*}
 \begin{split}
(\mathbf{A4})\;|\partial_t^4u(\mathbf{x},t)|\leq \mathcal{Q}, \; |\partial_t^3\partial_k^2u(\mathbf{x},t)|\leq \mathcal{Q},\; |\partial_k^6u(\mathbf{x},t)|\leq \mathcal{Q}, \;|\partial_t\partial_k^4u(\mathbf{x},t)|\leq \mathcal{Q},\; (\mathbf{x},t) \in \Gamma\times [0,T], \;  k=x,y.
  \end{split}
\end{equation*}

 \subsection{Stability and dissipative property}

Next, we will introduce the following lemma that will be used for the subsequent theoretical analysis.
\begin{lemma}\cite{Liao, Xie, Zhang}\label{lemma4.1}
  For any $\omega\in W_h$, we have $\frac{4}{9}\|\omega\|\leq \|\mathcal{H}\omega\|\leq \|\omega\|$.
\end{lemma}

\begin{lemma}\label{lemma4.2}
 For any $\omega\in W_h$,  we have
  \begin{equation}
    \begin{split}
    \|\omega\|_E\leq \frac{2\sqrt{3}}{3}\|\omega\|,\quad
    \frac{2}{3}\|\omega\|\leq \|\omega\|_F\leq \frac{4}{3}\|\omega\|.
   \end{split}
   \end{equation}
\end{lemma}
\begin{proof}
From \eqref{eq4.7}, we can obtain that
\begin{equation}\label{eq4.8}
   \begin{split}
\|\omega\|_E\leq \frac{2}{3}\sqrt{3h_1h_2\sum_{i=1}^{J_1}\sum_{j=1}^{J_2}\left(\omega^2_{2i-2,2j-2}+\omega^2_{2i-2,2j-1}+\omega^2_{2i-1,2j-2}+\omega^2_{2i-1,2j-1}\right)}=\frac{2\sqrt{3}}{3}\|\omega\|.
   \end{split}
   \end{equation}
Based on \eqref{eq4.8} and the definition of $\|\cdot\|_F$, the proof can be established.
\end{proof}
\begin{theorem}\label{theorem4.1} Let $\{U_{ij}^n|(x_i,y_j)\in \Gamma_h,\;1\leq n\leq N\}$ and $\{V_{ij}^n|(x_i,y_j)\in \Gamma_h,\;1\leq n\leq N\}$ be the solution of \eqref{eq4.10}-\eqref{eq4.13}. Assume that the assumption ($\mathbf{A4}$) holds, we have
  \begin{equation}\label{UU1}
   \begin{split}
          \|U^N\|_a\leq
          \|U^0\|_a+2\tau\sum_{j=1}^N\|f^j\|,
   \end{split}
   \end{equation}
where $\|U^n\|_a:=\sqrt{\|\mathcal{H}\delta_t {U}^{n+1}\|^2+\frac{1}{2}\left(\|\mathcal{H}V^{n+1}\|^2+\|\mathcal{H}V^{n}\|^2\right)}$.
\end{theorem}
\begin{proof} Take the inner product of \eqref{eq4.10} with $\mathcal{H}\bar{\delta_t} {U}^n$ and \eqref{eq4.11} with $\mathcal{H} \widetilde{V}^n$, respectively, then we have
\begin{equation}\label{eq4.14}
   \begin{split}
         \langle \mathcal{H}\delta_t^{(2)}{U}^n, \mathcal{H}\bar{\delta_t} {U}^n \rangle  +  P\left( \|V^n\|_F^2 \right) \|\mathcal{H}\bar{\delta_t} {U}^n\|^2 + \langle \Phi\widetilde{V}^n, \mathcal{H}\bar{\delta_t} {U}^n \rangle
         = \langle \mathcal{H}{f}^n, \mathcal{H}\bar{\delta_t} {U}^n \rangle,
   \end{split}
   \end{equation}
   \begin{equation}\label{eq4.15}
   \begin{split}
         \langle \mathcal{H}\bar{\delta_t}{V}^n, \mathcal{H} \widetilde{V}^n \rangle
         = \langle \Phi\bar{\delta_t}U^n, \mathcal{H} \widetilde{V}^n \rangle, \quad n \geq 1.
   \end{split}
   \end{equation}
 We then add \eqref{eq4.14} and \eqref{eq4.15} together to obtain that
\begin{equation}\label{eq4.16}
   \begin{split}
         \langle \mathcal{H}\delta_t^{(2)}{U}^n, \mathcal{H}\bar{\delta_t} {U}^n \rangle  +  P\left( \|V^n\|_F^2 \right) \|\mathcal{H}\bar{\delta_t} {U}^n\|^2 + \langle \mathcal{H}\bar{\delta_t}{V}^n, \mathcal{H} \widetilde{V}^n \rangle
         = \langle \mathcal{H}{f}^n, \mathcal{H}\bar{\delta_t} {U}^n \rangle,
   \end{split}
   \end{equation}
where
\begin{equation*}
   \begin{split}
        \langle \mathcal{H}\omega, \Phi v\rangle
         &= \langle \mathcal{A}\mathcal{B}\omega, \mathcal{B}\delta_{x}\delta_{\Bar{x}}v \rangle+\langle \mathcal{A}\mathcal{B}\omega, \mathcal{A}\delta_{y}\delta_{\Bar{y}}v \rangle=\langle \mathcal{A}(\mathcal{B}\omega), \delta_{x}\delta_{\Bar{x}}(\mathcal{B}v) \rangle+\langle \mathcal{B}(\mathcal{A}\omega), \delta_{y}\delta_{\Bar{y}}(\mathcal{A}v) \rangle\\
         &=\langle (I_x+\frac{h_1^2}{12}\delta_{x}\delta_{\Bar{x}})(\mathcal{B}\omega), \delta_{x}\delta_{\Bar{x}}(\mathcal{B}v) \rangle+\langle (I_y+\frac{h_2^2}{12}\delta_{y}\delta_{\Bar{y}})(\mathcal{A}\omega), \delta_{y}\delta_{\Bar{y}}(\mathcal{A}v) \rangle\\
         &=\langle \delta_{x}\delta_{\Bar{x}}(\mathcal{B}\omega), \mathcal{B}v \rangle+\frac{h_1^2}{12}\langle \delta_{x}\delta_{\Bar{x}}\mathcal{B}\omega, \delta_{x}\delta_{\Bar{x}}\mathcal{B}v\rangle+\langle \delta_{y}\delta_{\Bar{y}}(\mathcal{A}\omega), \mathcal{A}v \rangle+\frac{h_2^2}{12}\langle \delta_{y}\delta_{\Bar{y}}\mathcal{A}\omega, \delta_{y}\delta_{\Bar{y}}\mathcal{A}v\rangle\\
         &=\langle \delta_{x}\delta_{\Bar{x}}\mathcal{B}\omega, (I_x+\frac{h_1^2}{12}\delta_{x}\delta_{\Bar{x}})\mathcal{B}v \rangle+\langle \delta_{y}\delta_{\Bar{y}}\mathcal{A}\omega, (I_y+\frac{h_2^2}{12}\delta_{y}\delta_{\Bar{y}})\mathcal{A}v \rangle\\
         &=\langle \delta_{x}\delta_{\Bar{x}}\mathcal{B}\omega, \mathcal{A}\mathcal{B}v \rangle+\langle \delta_{y}\delta_{\Bar{y}}\mathcal{A}\omega, \mathcal{B}\mathcal{A}v \rangle=\langle (\delta_{x}\delta_{\Bar{x}}\mathcal{B}+\delta_{y}\delta_{\Bar{y}}\mathcal{A})\omega, \mathcal{A}\mathcal{B}v \rangle=\langle \Phi\omega, \mathcal{H} v\rangle
   \end{split}
   \end{equation*}
is applied. First, notice that
\begin{equation}\label{eq4.17}
   \begin{split}
        & \langle \mathcal{H}\delta_t^{(2)} {U}^n, \mathcal{H}\bar{\delta_t}{U}^n \rangle
        = \frac{1}{2\tau} \left[ \|\mathcal{H}\delta_t {U}^{n+1}\|^2 -  \|\mathcal{H}\delta_t {U}^{n}\|^2 \right], \\
        & \langle  \mathcal{H}\bar{\delta_t}{V}^n, \mathcal{H} \widetilde{V}^n \rangle
        = \frac{1}{4\tau}\left[ \| \mathcal{H}{V}^{n+1}\|^2 - \|\mathcal{H}{V}^{n-1}\|^2 \right].
   \end{split}
   \end{equation}
We then substitute \eqref{eq4.17} into \eqref{eq4.16} and apply ($\mathbf{A1}$) to arrive at
\begin{equation}\label{eq4.18}
   \begin{split}
        &\qquad\qquad\quad\frac{1}{2\tau} \left[ \|\mathcal{H}\delta_t {U}^{n+1}\|^2 -  \|\mathcal{H}\delta_t {U}^{n}\|^2 \right] + p_0\|\mathcal{H}\bar{\delta_t} {U}^n\|^2 \\
        &+\frac{1}{4\tau}\left[ \| \mathcal{H}{V}^{n+1}\|^2 +\| \mathcal{H}{V}^{n}\|^2 -\left(\| \mathcal{H}{V}^{n}\|^2+ \|\mathcal{H}{V}^{n-1}\|^2 \right)\right] \leq \|\mathcal{H}{f}^n\| \|\mathcal{H}\bar{\delta_t} {U}^n\|.
   \end{split}
   \end{equation}
We sum \eqref{eq4.18} with respect to $n$ from $1$ to $N$ and multiply the resulting equation by $2\tau $ to obtain 
\begin{equation}\label{eq4.19}
   \begin{split}
        &\qquad\qquad \qquad\quad \|\mathcal{H}\delta_t {U}^{N+1}\|^2 - \|\mathcal{H}\delta_t {U}^{1}\|^2 + 2p_0\tau\sum_{n=1}^N\|\mathcal{H}\bar{\delta_t} {U}^n\|^2 \\
         &+\frac{1}{2}\left[ \left(\|\mathcal{H} {V}^{N+1}\|^2 +\|\mathcal{H} {V}^{N}\|^2\right) -\left(\|\mathcal{H}{V}^{1}\|^2 + \|\mathcal{H}{V}^{0}\|^2 \right)\right]\leq 2\tau\sum_{n=1}^N\|\mathcal{H}{f}^n\| \|\mathcal{H}\bar{\delta_t} {U}^n\|.
   \end{split}
   \end{equation}
By Lemma \ref{lemma4.1} and the definition of $\|\cdot\|_a$ in the line below \eqref{UU1}, \eqref{eq4.19} can be expressed as
\begin{equation}\label{eq4.21}
   \begin{split}
    \|U^N\|_a^2+2p_0\tau\sum_{n=1}^N\|\mathcal{H}\bar{\delta_t} {U}^n\|^2\leq \|U^0\|_a^2+2\tau\sum_{n=1}^N\|{f}^n\| \|\mathcal{H}\bar{\delta_t} {U}^n\|.
   \end{split}
   \end{equation}
Then we have
\begin{equation}\label{eq4.22}
   \begin{split}
        &\|\mathcal{H}\bar{\delta_t} {U}^n\|=\left\|\frac{\mathcal{H}U^{n+1}-\mathcal{H}U^{n}+\mathcal{H}U^{n}-\mathcal{H}U^{n-1}}{2\tau}\right\|\leq \frac{\|\mathcal{H}\delta_t U^{n+1}\|+\|\mathcal{H}\delta_t U^{n}\|}{2}\leq \frac{\| U^{n}\|_a+\|U^{n-1}\|_a}{2}.
   \end{split}
   \end{equation}
We then substitute \eqref{eq4.22} into \eqref{eq4.21} to obtain that
\begin{equation}\label{eq4.23}
   \begin{split}
       \|U^N\|_a^2 \leq \|U^0\|_a^2+2\tau\sum_{n=1}^N\|{f}^n\| \frac{\| U^{n}\|_a+\|U^{n-1}\|_a}{2}.
   \end{split}
   \end{equation}
By taking the suitable $m$ such that $\|U^m\|_a=\max\limits_{0\leq n \leq N}\|U^n\|_a$, we get
\begin{equation*}
   \begin{split}
              \|U^N\|_a \leq \|U^m\|_a \leq  \|U^0\|_a+2\tau\sum_{n=1}^N\|{f}^n\|.
   \end{split}
   \end{equation*}
This completes finish the proof.
\end{proof}

We then prove the energy dissipative property of the fully discrete compact difference scheme \eqref{eq4.10}-\eqref{eq4.13} in the following theorem.

\begin{corollary}\label{coro2}
    The fully discrete compact difference scheme \eqref{eq4.10}-\eqref{eq4.13} satisfies the energy dissipative property as follows
    \begin{equation*}
         0\leq\check{\mathbf{E}}^n \leq \check{\mathbf{E}}^0, \quad 1\leq n \leq N,
    \end{equation*}
    where
    \begin{equation}\label{energy2d}
\begin{split}
     \check{\mathbf{E}}^n := \sqrt{\|\mathcal{H}\delta_t {U}^{n+1}\|^2+\frac{1}{2}\left(\|\mathcal{H}V^{n+1}\|^2+\|\mathcal{H}V^{n}\|^2\right)}.
\end{split}
\end{equation}
\end{corollary}
\begin{proof}
  We set $f=0$. The proof could be similarly carried out by following Theorem \ref{theorem4.1}  and is thus omitted.
\end{proof}

\subsection{Convergence analysis}
For $0\leq i\leq 2J_1$, $0\leq j\leq 2J_2$ and $ 1\leq n\leq N$, we have
\begin{equation}\label{eq4.24}
   \begin{split}
        \mathcal{H}\delta_t^{(2)} u_{ij}^n &+ P\left(\int_{\Gamma} \left|  v(\mathbf{x},t_n)\right|^2 d\mathbf{x} \right) \mathcal{H}\bar{\delta_t} u_{ij}^n + \Phi\widetilde{v}_{ij}^n = \mathcal{H}{f}_{ij}^n + (\mathbb{R}_1)_{ij}^n + \sum\limits_{m=3}^{5}(\mathbb{R}_m)_{ij}^n,
   \end{split}
   \end{equation}
   \begin{equation}\label{eq4.25}
   \begin{split}
        \mathcal{H}\bar{\delta_t}v_{ij}^n = \Phi\bar{\delta_t}u_{ij}^n + (\mathbb{R}_2)_{ij}^n,
   \end{split}
   \end{equation}
\begin{equation}\label{eq4.26}
   \begin{split}
        & u_{0,j}^n = u_{2J_1,j}^n = u_{i,0}^n = u_{i,2J_2}^n = 0, \quad v_{0,j}^n = v_{2J_1,j}^n = v_{i,0}^n = v_{i,2J_2}^n = 0,
   \end{split}
   \end{equation}
\begin{equation}\label{eq4.27}
   \begin{split}
    u_{ij}^0 = u_0(x_i,y_j), \quad (\partial_tu)_{ij}^0 = u_1(x_i,y_j),
   \end{split}
\end{equation}
where the truncation errors
\begin{equation}\label{eq4.28}
   \begin{split}
        & (\mathbb{R}_1)_{ij}^n = \mathcal{H}{(\Delta v_{ij}^n)}-\Phi v_{ij}^n,\\
        & (\mathbb{R}_2)_{ij}^n = (\mathcal{R}_2)_{ij}^n+\mathcal{H}\left(\bar{\delta_t} v_{ij}^n - v_{ij}'(t_n)\right)+\Phi\left(\bar{\delta_t} u_{ij}^n - u_{ij}'(t_n)\right),  \\
        & (\mathbb{R}_3)_{ij}^n = \mathcal{H}\left(\delta_t^{(2)} u_{ij}^n - u_{ij}''(t_n)\right), \\
        & (\mathbb{R}_4)_{ij}^n = P\left( \int_{\Gamma} \left|  v(\mathbf{x},t_n)\right|^2 d\mathbf{x} \right) \mathcal{H}\left[\bar{\delta_t} u_{ij}^n - u_{ij}'(t_n) \right], \\
        &  (\mathbb{R}_5)_{ij}^n = \Phi\left( \widetilde{v}_{ij}^n-v(x_i, y_j, t_n)\right),
   \end{split}
   \end{equation}
with $|(\mathcal{R}_2)_{ij}^n|:=|\mathcal{H}v_{ij}^n-\Phi u_{ij}^n|$.
 Then, define $\xi_{ij}^n := u_{ij}^n - U_{ij}^n$ and $ \eta_{ij}^n:=v_{ij}^n-V_{ij}^n$ for $0\leq n \leq N$, $0\leq i\leq 2J_1$ and $0\leq j\leq 2J_2$. By subtracting \eqref{eq4.10}-\eqref{eq4.13} from \eqref{eq4.24}-\eqref{eq4.27}, we obtain the following error equations
\begin{equation}\label{eq4.29}
   \begin{split}
        \mathcal{H}\delta_t^{(2)}\xi_{ij}^n &+ P\left( \|{V}^n\|_F^2 \right) \mathcal{H}\bar{\delta_t} \xi_{ij}^n  + \Phi\widetilde{\eta}_{ij}^n
         =  (\mathbb{R}_1)_{ij}^n \\
         &+\left[  P\left( \|V^n\|_F^2 \right) -  P\left(
        \int_{\Gamma} \left|  v(\mathbf{x},t_n)\right|^2 d\mathbf{x} \right)  \right] \mathcal{H}\bar{\delta_t}u_{ij}^n + \sum\limits_{m=3}^{5}(\mathbb{R}_m)_{ij}^n,
   \end{split}
   \end{equation}
   \begin{equation}\label{eq4.30}
   \begin{split}
        \mathcal{H}\bar{\delta_t} \eta_{ij}^n
         = \Phi\bar{\delta_t}\xi_j^n+(\mathbb{R}_2)_{ij}^n,
   \end{split}
   \end{equation}
\begin{equation}\label{eq4.31}
   \begin{split}
        & \xi_{0,j}^n = \xi_{2J_1,j}^n = \xi_{i,0}^n = \xi_{i,2J_2}^n = 0, \quad \eta_{0,j}^n = \eta_{2J_1,j}^n = \eta_{i,0}^n = \eta_{i,2J_2}^n = 0,
   \end{split}
   \end{equation}
\begin{equation}\label{eq4.32}
   \begin{split}
        \xi_{ij}^0 = 0, \quad \delta_t \xi_{ij}^1 = \delta_tu_{ij}^1 - \delta_t U^1_{ij} = \frac{1}{2\tau} \int_{0}^{t_1} \partial^3_tu(x_i,y_j,s)(s-t_1)^2ds.
   \end{split}
   \end{equation}
We then prove the following  convergence result.

\begin{theorem} \label{theorem4.2}
Let  $\{u_{ij}^n|(x_i,y_j)\in \Gamma_h,\;1\leq n\leq N\}$ and $\{v_{ij}^n|(x_i,y_j)\in \Gamma_h,\;1\leq n\leq N\}$ be the solution of \eqref{eq4.24}-\eqref{eq4.27}. Suppose that the assumption  ($\mathbf{A4}$) holds,
then it follows that
\begin{equation}\label{qqq11}
 \begin{split}
    \sqrt{\|\delta_t({u}^{n+1}-{U}^{n+1})\|^2
       + \|{v}^{n+1} -{V}^{n+1} \|^2 + \|{v}^{n} -{V}^{n} \|^2}  \leq \mathcal{Q}(T) (\tau^2 + h^4),\quad 1\leq n \leq N.
 \end{split}
\end{equation}
\end{theorem}
\begin{proof} Firstly, we take the inner product of \eqref{eq4.29} and \eqref{eq4.30} with $\mathcal{H}\bar{\delta_t}{\xi}^n$ and $\mathcal{H}\widetilde{\eta}^n$, respectively, and add the two resulting equations and apply the assumption ($\mathbf{A1}$) to obtain
\begin{equation}\label{eq4.34}
   \begin{split}
       & \langle\mathcal{H}\delta_t^{(2)}\xi^n,  \mathcal{H}\bar{\delta_t}{\xi}^n\rangle + p_0 \|\mathcal{H}\bar{\delta_t}{\xi}^n\|^2  + \langle\mathcal{H}\bar{\delta_t}\eta^n,  \mathcal{H}\widetilde{\eta}^n\rangle
         \leq \langle({\mathbb{R}}_1)^n, \mathcal{H}\bar{\delta_t}{\xi}^n\rangle + \langle({\mathbb{R}}_2)^n, \mathcal{H}\widetilde{\eta}^n\rangle\\
       & + \sum\limits_{m=3}^{5}\langle({\mathbb{R}}_m)^n, \mathcal{H}\bar{\delta_t}{\xi}^n\rangle  + \left[  P\left( \|{V}^n\|_F^2 \right) -  P\left(
        \int_{\Gamma} \left|  v(\mathbf{x},t_n)\right|^2 d\mathbf{x} \right)  \right] \langle\mathcal{H}\bar{\delta_t}{u}^n, \mathcal{H}\bar{\delta_t}{\xi}^n\rangle.
   \end{split}
   \end{equation}
After calculation, we have
\begin{equation}\label{eq4.35}
   \begin{split}
        & \langle \mathcal{H}\delta_t^{(2)} {\xi}^n, \mathcal{H}\bar{\delta_t}{\xi}^n \rangle
        = \frac{1}{2\tau} \left[ \|\mathcal{H}\delta_t {\xi}^{n+1}\|^2 -  \|\mathcal{H}\delta_t {\xi}^{n}\|^2 \right], \\
        & \langle \mathcal{H}\bar{\delta_t}{\eta}^n, \mathcal{H} \widetilde{\eta}^n \rangle
        = \frac{1}{4\tau}\left[ \| \mathcal{H}{\eta}^{n+1}\|^2 - \|\mathcal{H}{\eta}^{n-1}\|^2 \right].
   \end{split}
   \end{equation}
Similar to the analysis of \eqref{eq2.27}-\eqref{eq2.30}, we utilize the Cauchy-Schwarz inequality, Lemma \ref{lemma4.2} and \eqref{eq4.35} to deduce that
\begin{equation}\label{eq4.36}
   \begin{split}
       & \frac{1}{2\tau}\left[\|\mathcal{H}\delta_t\xi^{n+1}\|^2-\|\mathcal{H}\delta_t\xi^{n}\|^2\right] + p_0 \|\mathcal{H}\bar{\delta_t}{\xi}^n\|^2 \\
       & + \frac{1}{4\tau}\left[\|\mathcal{H}\eta^{n+1}\|^2+\|\mathcal{H}\eta^{n}\|^2-\left(\|\mathcal{H}\eta^{n}\|^2+\|\mathcal{H}\eta^{n-1}\|^2\right) \right]\\
       & \leq \sum\limits_{m=3}^{5}\|({\mathbb{R}}_m)^n\|\|\mathcal{H}\bar{\delta_t}{\xi}^n\| + \|({\mathbb{R}}_1)^n\|\|\mathcal{H}\bar{\delta_t}{\xi}^n\|\\
       & + \|({\mathbb{R}}_2)^n\| \|\mathcal{H}\widetilde{\eta}^n\|+ \mathcal{Q}\left(h_1^4+h_2^4+\|\eta^n\|\right)\|\mathcal{H}\bar{\delta_t}\xi^n\|,
   \end{split}
   \end{equation}
where
\begin{equation*}
    \begin{split}
    |\bar{\delta_t}u_{ij}^n|=&\Bigg|\partial_tu(x_i,y_j,t_n)+\frac{1}{2\tau}\left(\int_{t_n}^{t_{n+1}}(t_n-s)\partial_s^2u(x_i,y_j,s)ds-\int_{t_{n-1}}^{t_{n}}(s-t_{n-1})\partial_s^2u(x_i,y_j,s)ds\right)\Bigg|
          \leq \mathcal{Q},
    \end{split}
\end{equation*}
and thus implies that $\|\bar{\delta_t}u^n\|\leq \mathcal{Q}$. Subsequently, we define
\begin{equation}\label{eq4.37}
    \begin{split}
\|\xi^n\|_b:=\sqrt{\|\mathcal{H}\delta_t\xi^{n+1}\|^2+\frac{1}{2}\left(\|\mathcal{H}\eta^{n+1}\|^2+\|\mathcal{H}\eta^{n}\|^2\right)}.
    \end{split}
\end{equation}
We sum \eqref{eq4.36} for $n$ from $1$ to $N$ and then multiply the resulting inequality by $2\tau$ to find that
\begin{equation}\label{eq4.38}
   \begin{split}
       & \|\xi^N\|_b^2 -  \|\xi^0\|_b^2 +2\tau p_0 \sum_{n=1}^N\|\mathcal{H}\bar{\delta_t}{\xi}^n\|^2 \leq 2\tau\sum_{n=1}^N\sum\limits_{m=3}^{5}\|({\mathbb{R}}_m)^n\|\|\mathcal{H}\bar{\delta_t}{\xi}^n\| \\
       & + 2\tau\sum_{n=1}^N\|({\mathbb{R}}_1)^n\|\|\mathcal{H}\bar{\delta_t}{\xi}^n\| + 2\tau\sum_{n=1}^N\|({\mathbb{R}}_2)^n\| \|\mathcal{H}\widetilde{\eta}^n\|+ \mathcal{Q}\tau\sum_{n=1}^N\left(h_1^4+h^4_2+\|\eta^n\|\right)\|\mathcal{H}\bar{\delta_t}\xi^n\|.
   \end{split}
   \end{equation}
We employ the relations
\begin{equation}\label{eq4.39}
   \begin{split}
        &\|\mathcal{H}\bar{\delta_t} {\xi}^n\|
        \leq \frac{\| \xi^{n}\|_b+\|\xi^{n-1}\|_b}{2},
   \end{split}
   \end{equation}
and
\begin{equation}\label{eq4.40}
   \begin{split}
        &\|\mathcal{H}\widetilde{\eta}^n\|
        \leq \frac{\sqrt{2}}{2}\left(\| \xi^{n}\|_b+\|\xi^{n-1}\|_b\right),
   \end{split}
   \end{equation}
and choose the suitable $m$ such that $\|\xi^m\|_b=\max\limits_{0\leq n\leq N}\|\xi^n\|_b$ to conclude that
\begin{equation}\label{eq4.41}
   \begin{split}
        \|\xi^m\|_b^2 \leq \|\xi^0\|_b\|\xi^m\|_b + \mathcal{Q}\tau\sum_{n=1}^N\sum\limits_{m=1}^{5}\|({\mathbb{R}}_m)^n\|\|\xi^m\|_b +  \mathcal{Q}\tau\sum_{n=1}^N\left(h_1^4+h^4_2+\|\xi^n\|_b\right)\|\xi^m\|_b.
   \end{split}
   \end{equation}
Hence,
   \begin{equation}\label{g1eq4.41}
   \begin{split}
      \|\xi^N\|_b \leq \|\xi^m\|_b \leq \|\xi^0\|_b + \mathcal{Q}\tau\sum_{n=1}^N\sum\limits_{m=1}^{5}\|({\mathbb{R}}_m)^n\| +  \mathcal{Q}\tau\sum_{n=1}^N\left(h_1^4+h^4_2+\|\xi^n\|_b\right).
   \end{split}
   \end{equation}
After that, we will examine $\|\xi^0\|_b$. First, we utilize $\xi^0=0$ to arrive at
\begin{equation}\label{eq4.42}
   \begin{split}
        \xi^1_{ij}=u^1_{ij}-U^1_{ij}=\frac{1}{2}\int_{0}^{t_1}(s-t_1)^2\partial_s^3u(x_i,y_j,s)ds \Rightarrow \|\mathcal{H}\delta_t\xi^1\|\leq \mathcal{Q}\tau^2.
   \end{split}
   \end{equation}
Note that $\eta^0=0 $, we similarly can get
\begin{equation}\label{eq4.43}
   \begin{split}
         \eta^1_{ij}=v^1_{ij}-V^1_{ij}=\frac{1}{2}\int_{0}^{t_1}(s-t_1)^2\partial_s^3v(x_i,y_j,s)ds \Rightarrow \|\mathcal{H}\eta^1\|\leq \mathcal{Q}\tau^2.
   \end{split}
   \end{equation}
As a result,
\begin{equation}\label{eq4.44}
   \begin{split}
        \|\xi^0\|_b
              \leq \mathcal{Q}\tau^2.
   \end{split}
   \end{equation}
We leverage the Taylor expansion with integral remainder to obtain
\begin{equation*}
   \begin{split}
          &u''_{ij}(t_n)-\delta_t^{(2)}u^n_{ij}=\frac{-1}{6\tau^2}\left[\int_{t_n}^{t_{n+1}}(t_{n+1}-s)^3\partial_s^4u(x_i,y_j,s)ds+\int_{t_{n-1}}^{t_{n}}(s-t_{n-1})^3\partial_s^4u(x_i,y_j,s)ds\right],\\
          &n\geq 2, \quad u''_{ij}(t_1)-\delta_t^{(2)}u^1_{ij}=\frac{-1}{2\tau^2}\left[\int_{t_1}^{t_{2}}(t_{2}-s)^2\partial_s^3u(x_i,y_j,s)ds+\int_{0}^{t_{1}}s^2\partial_s^3u(x_i,y_j,s)ds\right],
   \end{split}
   \end{equation*}
consequently, we conclude that
\begin{equation}\label{eq4.46}
   \begin{split}
          \tau\sum_{n=1}^N\|({\mathbb{R}}_3)^n\|\leq\tau\int_0^{2\tau}\|\partial_s^3u(s)\|ds+\tau^2\int_\tau^{t_{N+1}}\|\partial_s^4u(s)\|ds \leq \mathcal{Q}\tau^2.
   \end{split}
   \end{equation}
As for $({\mathbb{R}}_4)_j^n $, we then have
\begin{equation*}
   \begin{split}
    u'_{ij}(t_n)-\bar{\delta_t}u^n_{ij}=\frac{-1}{4\tau}\left[\int_{t_n}^{t_{n+1}}(t_{n+1}-s)^2\partial_s^3u(x_i,y_j,s)ds+\int_{t_{n-1}}^{t_{n}}(s-t_{n-1})^2\partial_s^3u(x_i,y_j,s)ds\right],
   \end{split}
   \end{equation*}
as a result, we further comprehend that
\begin{equation}\label{eq4.47}
   \begin{split}
    \tau\sum_{n=1}^N\|({\mathbb{R}}_4)^n\|\leq \frac{p_0}{2}\tau^2\int_0^{t_{N+1}}\|\partial_s^3u(s)\|ds \leq \mathcal{Q}\tau^2.
   \end{split}
   \end{equation}
From \eqref{eq4.28} and Lemma \ref{lemma2.3}, we can find that
\begin{equation}\label{eq4.45}
   \begin{split}
          \tau\sum_{n=1}^N\sum\limits_{m=1}^{2}\|({\mathbb{R}}_m)^n\|\leq \mathcal{Q}\left(h_1^4+h_2^4\right).
   \end{split}
   \end{equation}
We finally analyze the term $({\mathbb{R}}_5)_j^n$ to find that
\begin{equation*}
   \begin{split}
          v_{ij}^n-\widetilde{v}_{ij}^n=\frac{-1}{2}\left[\int_{t_n}^{t_{n+1}}(t_{n+1}-s)\partial_s^2v(x_i,y_j,s)ds+\int_{t_{n-1}}^{t_{n}}(s-t_{n-1})\partial_s^2v(x_i,y_j,s)ds\right],
   \end{split}
   \end{equation*}
accordingly,
\begin{equation}\label{eq4.48}
   \begin{split}
          \tau\sum_{n=1}^N\|({\mathbb{R}}_5)^n\|\leq \tau^2\int_0^{t_{N+1}}\|\partial_s^2v(s)\|ds \leq \mathcal{Q}\tau^2.
   \end{split}
   \end{equation}
Substituting \eqref{eq4.44}-\eqref{eq4.48} into \eqref{eq4.41} and utilizing the discrete Gr\"onwall's lemma, we can yield \eqref{qqq11}. This completes the proof.
\end{proof}

\section{Numerical experiment}\label{sec5}
 We present numerical examples to investigate the convergence behavior of fully discrete compact difference schemes \eqref{eq3.1}-\eqref{eq3.4} and \eqref{eq4.10}-\eqref{eq4.13}.

\begin{table}
    \center \small
    \caption{Discrete $L^2$-norm errors and convergence rates for Example 1.} \label{tab1}
    \begin{tabular}{cccccccccccc}
      \toprule
      & \multicolumn{2}{c}{$J=2^6$, $P(\omega)=\sqrt{1+\omega}$} & &\multicolumn{2}{c}{$N=2^{15}$, $P(\omega)=\sqrt{1+\omega}$}\\
       \cmidrule(lr){2-3}\cmidrule(lr){5-6}
       $N$  & $F_U(\tau,h)$ & $Order_U^t$ & $2J$  & $G_U(\tau,h)$ & $Order_U^h$\\
      \midrule
        $2^7$  & $2.3724 \times 10^{-3}$   &    *    &  $2^3$   & $4.3416 \times 10^{-4}$  &  *  \\
        $2^8$  & $6.3628 \times 10^{-4}$   &  1.90   &  $2^4$  & $2.7938 \times 10^{-5}$  &  3.96 \\
        $2^9$  & $1.6217 \times 10^{-4}$   &  1.97   &  $2^5$  & $1.7522 \times 10^{-6}$  &  4.00 \\
        $2^{10}$ & $4.0782 \times 10^{-5}$   &  1.99   &  $2^6$  & $1.1173 \times 10^{-7}$  &  3.97 \\
        \text{Theory} &     &  2.00  &          &      &    4.00  &   \\
        \midrule
        & \multicolumn{2}{c}{$J=2^6$, $P(\omega)=1+\omega$} & &\multicolumn{2}{c}{$N=2^{15}$, $P(\omega)=1+\omega$}\\
       \cmidrule(lr){2-3}\cmidrule(lr){5-6}
       $N$  & $F_U(\tau,h)$ & $Order_U^t$ & $2J$  & $G_U(\tau,h)$ & $Order_U^h$\\
      \midrule
        $2^7$  & $7.3401 \times 10^{-4}$   &    *    &  $2^3$  & $1.2399 \times 10^{-3}$  &  *  \\
        $2^8$  & $2.0381 \times 10^{-4}$   &  1.85   &  $2^4$  & $7.6212 \times 10^{-5}$  &  4.02 \\
        $2^9$  & $5.3179 \times 10^{-5}$   &  1.94   &  $2^5$  & $4.7695 \times 10^{-6}$  &  4.00 \\
        $2^{10}$ & $1.3566 \times 10^{-5}$ &  1.97   &  $2^6$  & $3.0480 \times 10^{-7}$  &  3.97 \\
        \text{Theory} &     &  2.00  &          &      &    4.00  &   \\
      \bottomrule
    \end{tabular}
\end{table}

\textbf{Example 1. (1D case)}   Let $\Gamma = (0, 1)$, $T=1$, $f(x,t)=t^3\sin(\pi x)$, $u_0(x)=\sin(\pi x)$ and $u_1(x)=0$. We fix the spatial partition $J$ to investigate the temporal convergence behavior of the scheme and fix the temporal partition $N$ to test its spatial convergence rates.
Since the exact solution is unknown, we define the temporal and spatial $L^2$-norm errors as follows
\begin{equation}\label{FGU}
\begin{split}
     F_U(\tau,h):=\sqrt{h\sum_{j=1}^{2J-1}\left|U_j^{N+1}-U_j^{2N+1}\right|^2},\quad
     G_U(\tau,h):=\sqrt{h\sum_{j=1}^{2J-1}\left|U_j^{N+1}-U_{2j}^{N+1}\right|^2},
\end{split}
\end{equation}
and accordingly define the spatial and temporal convergence rates
\begin{equation}\label{oFG}
\begin{split}
     Order_U^t=\log_{2}\left(\frac{F_U(2\tau,h)}{F_U(\tau,h)}\right),\quad
     Order_U^{h}=\log_{2}\left(\frac{G_{U}(\tau,2h)}{G_{U}(\tau,h)}\right),
\end{split}
\end{equation}
We present the numerical results in Table \ref{tab1}, which illustrates the second-order accuracy in time and the fourth-order accuracy in space of the scheme \eqref{eq3.1}-\eqref{eq3.4} as proved in Theorem \ref{theorem3.2}.

 In the left plot of Fig.~\ref{Figerror1}, we  present  the plot of the energy \eqref{energy1d} by choosing the same data as those for Table \ref{tab1} except that  $f=0$, $J=2^4$, $N=2^{15}$ and $P(\omega)=\sqrt{1+\omega}$, which validates the energy dissipation property of the proposed scheme as proved in Corollary \ref{coro1}.

 \textbf{Example 2. (2D case)} Let $\Gamma = (0, 1)^2$, $T=1$, $f(x,y,t)=t^3\sin(\pi x)\sin(\pi y)$, $u_0(x,y)=\sin(\pi x)\sin(\pi y)$, $u_1(x,y)=0$, and $J_1=J_2=J$.
Since the exact solution is unknown, we define the following $L^2$-norm errors
\begin{equation*}
\begin{split}
     F_U^2(\tau,h):=h\sqrt{\sum_{i=1}^{2J-1}\sum_{j=1}^{2J-1}\left|U_{ij}^{N+1}-U_{ij}^{2N+1}\right|^2}, \quad
     G_U^2(\tau,h):=h\sqrt{\sum_{i=1}^{2J-1}\sum_{j=1}^{2J-1}\left|U_{ij}^{N+1}-U_{2i,2j}^{N+1}\right|^2}.
\end{split}
\end{equation*}
The spatial and temporal convergence rates can thus be calculated by
\begin{equation*}
\begin{split}
     Order_U^{2,t}=\log_{2}\left(\frac{F^2_U(2\tau,h)}{F^2_U(\tau,h)}\right),\quad
     Order_U^{2,h}=\log_{2}\left(\frac{G^2_{U}(\tau,2h)}{G^2_{U}(\tau,h)}\right).
\end{split}
\end{equation*}

\begin{figure}
\centering
\subfigure{
\begin{minipage}[t]{0.51\linewidth}
\centering
\includegraphics[width=\linewidth]{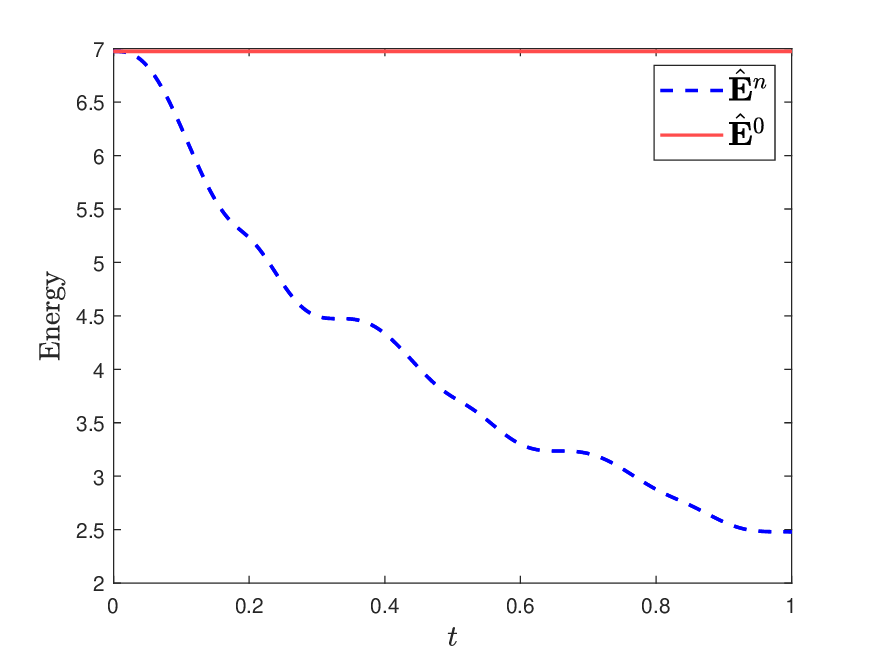}
\end{minipage}%
}%
\subfigure{
\begin{minipage}[t]{0.51\linewidth}
\centering
\includegraphics[width=\linewidth]{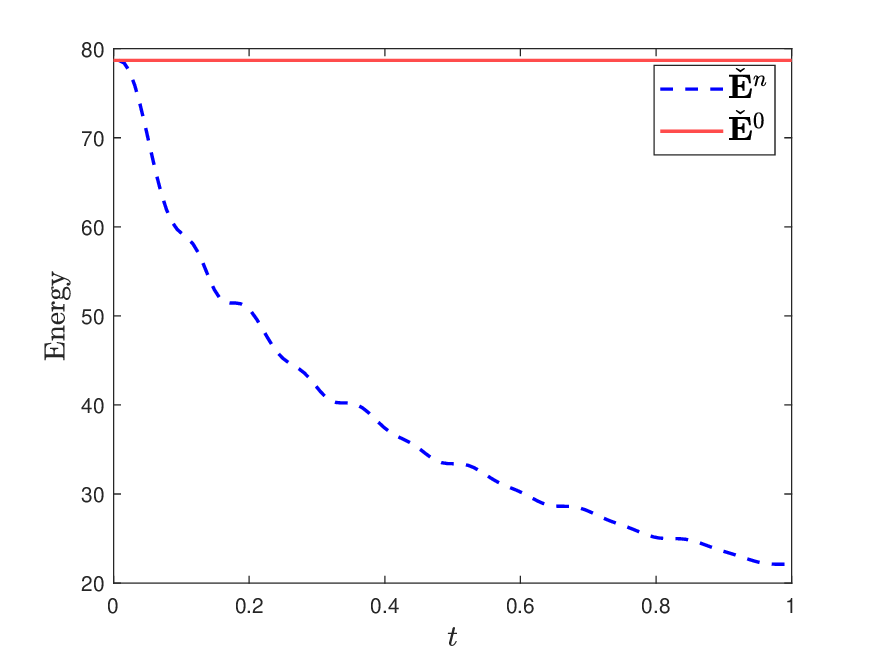}
\end{minipage}%
}%
\centering
\caption{The plot of the energy $\hat{\mathbf{E}}^n$ in Example 1 (left) and in Example 2 (right).}
\label{Figerror1}
\end{figure}

\begin{table}
    \center \small
    \caption{Discrete $L^2$-norm errors and convergence rates for Example 2.} \label{tab4}
    \begin{tabular}{cccccccccccc}
      \toprule
      & \multicolumn{2}{c}{$J=2^4$, $P(\omega)=1+\omega$} & &\multicolumn{2}{c}{$N=10000$, $P(\omega)=1+\omega$}\\
       \cmidrule(lr){2-3}\cmidrule(lr){5-6}
       $N$  & $F_U^2(\tau,h)$ & $Order_U^{2,t}$ & $2J$  & $G_U^2(\tau,h)$ & $Order_U^{2,h}$\\
      \midrule
        $2^8$     & $4.2355 \times 10^{-3}$   &    *    &  $2^3$  & $9.1955 \times 10^{-4}$  &  *  \\
        $2^9$     & $1.1489 \times 10^{-3}$   &  1.88   &  $2^4$  & $5.8554 \times 10^{-5}$  &  3.97 \\
        $2^{10}$  & $2.9374 \times 10^{-4}$   &  1.97   &  $2^5$  & $3.7624 \times 10^{-6}$  &  3.96  \\
        $2^{11}$  & $7.4107 \times 10^{-5}$   &  1.99   &  $2^6$  & $2.6328 \times 10^{-7}$  &  3.84 \\
        \text{Theory} &     &  2.00  &          &      &    4.00  &   \\
        \midrule
        & \multicolumn{2}{c}{$J=2^4$, $P(\omega)=\sqrt{1+\omega}$} & &\multicolumn{2}{c}{$N=10000$, $P(\omega)=\sqrt{1+\omega}$}\\
       \cmidrule(lr){2-3}\cmidrule(lr){5-6}
       $N$  & $F_U^2(\tau,h)$ & $Order_U^{2,t}$ & $2J$  & $G_U^2(\tau,h)$ & $Order_U^{2,h}$\\
      \midrule
        $2^8$     & $1.0497 \times 10^{-2}$   &    *    &  $2^3$  & $2.1823 \times 10^{-3}$  &  * \\
        $2^9$     & $2.7594 \times 10^{-3}$   &  1.93   &  $2^4$  & $1.3711 \times 10^{-4}$  &  3.99\\
        $2^{10}$  & $7.0018 \times 10^{-4}$   &  1.98   &  $2^5$  & $8.8033 \times 10^{-6}$  &  3.96 \\
        $2^{11}$  & $1.7631 \times 10^{-4}$   &  1.99   &  $2^6$  & $6.1602 \times 10^{-7}$  &  3.84 \\
        \text{Theory} &     &  2.00  &          &      &    4.00  &   \\
      \bottomrule
    \end{tabular}
\end{table}


 In Table \ref{tab4}, we test the $ L^2$-norm errors and convergence rates of the fully discrete scheme \eqref{eq4.10}-\eqref{eq4.13} for the 2D case. We observe that the numerical results are consistent with our theoretical findings (see Theorem \ref{theorem4.2}).
 In the right plot of Fig.~\ref{Figerror1}, we  present  the plot of the energy \eqref{energy2d} by choosing the same data as those for Table \ref{tab4} except that  $f=0$, $J=2^5$, $N=2^{7}$ and $P(\omega)= {1+\omega}$, which validates the energy dissipation property of the proposed scheme as proved in Corollary \ref{coro2}.

\section*{Acknowledgement}
This work was partially supported by the National Natural Science Foundation of China (No. 12301555), the Taishan Scholars Program of Shandong Province (No. tsqn202306083), and the National Key R\&D Program of China (No. 2023YFA1008903), the Postdoctoral Fellowship Program of CPSF (No. GZC20240938), and the China Postdoctoral Science Foundation (No. 2024M762459) and the Natural Science Foundation of Hubei Province (No. JCZRQN202500278).



\begin{thebibliography}{10}

  \bibitem{Ahn}  {\sc J. Ahn, D.E. Stewart}, {\em An Euler--Bernoulli Beam with Dynamic Contact: Discretization, Convergence, and Numerical Results}, SIAM J. Numer. Anal., 43 (2005), pp.~1455-1480.


  \bibitem{Ammari1}  {\sc K. Ammari, M. Choulli}, {\em Logarithmic stability in determining two coefficients in a dissipative wave equation. Extensions to clamped Euler--Bernoulli beam and heat equations}, J. Differ. Equations, 259 (2015), pp.~3344--3365.


  \bibitem{Bartolomeo} {\sc J. Bartolomeo, R. Triggiani}, {\em  Uniform energy decay rates for Euler--Bernoulli equations with feedback operators in the Dirichlet/Neumann boundary conditions}, SIAM J. Math. Anal., 22 (1991), pp.~46--71.

  \bibitem{Bauchau}  {\sc O. A. Bauchau, J. I. Craig}, {\em Euler--Bernoulli beam theory}, Springer, 2009.

  \bibitem {Cannarsa} {\sc P. Cannarsa, D. Sforza}, {\em A stability result for a class of nonlinear integro-differential equations with $L^1$ kernels}, Appl. Math., 35 (2008), pp.~395--430.

  \bibitem{Cavalcanti} {\sc M. M. Cavalcanti}, {\em Existence and uniform decay for the Euler--Bernoulli viscoelastic equation with nonlocal boundary dissipation}, Disc. Contin. Dyn. Syst., 8 (2002), pp.~675--696.

  \bibitem {Cavalcanti1} {\sc M. M. Cavalcanti, V. N. Domingos Cavalcanti, T. F. Ma}, {\em  Exponential decay of the viscoelastic Euler--Bernoulli equation with a nonlocal dissipation in general domains}, Differ. Integral Equ., 17 (2004), pp.~495--510.

  \bibitem {Chapra} {\sc S. C. Chapra}, {\em Applied numerical methods with MATLAB for engineers and scientists}, Mcgraw-hill, 2018.

  \bibitem {ChapraCanale} {\sc S. C. Chapra, R. P. Canale}, {\em Numerical methods for engineers}, Mcgraw-hillNew York, 2011.

  \bibitem{Emm} {\sc E. Emmrich, M. Thalhammer}, {\em A class of integro-differential equations incorporating nonlinear and nonlocal damping with applications in nonlinear elastodynamics: Existence via time discretization}, Nonlinearity, 24 (2011), pp.~2523--2546.


  \bibitem{Guo} {\sc B. Guo}, {\em Riesz Basis Property and Exponential Stability of Controlled Euler--Bernoulli Beam Equations with Variable Coefficients}, SIAM J. Control Optim., 40 (2002), pp.~1905--1923.


  \bibitem{Hasanov} {\sc A. Hasanov, A. Romanov, O. Baysal}, {\em Unique recovery of unknown spatial load in damped Euler--Bernoulli beam equation from final time measured output}, Inverse Probl., 37 (2021), p.~075005.

   \bibitem{Horn} {\sc M. A. Horn}, {\em Uniform decay rates for the solutions to the Euler--Bernoulli plate equation with boundary feedback acting via bending moments}, Differ. Integral Equ., 5 (1992), pp.~1121--1150.

    \bibitem{Hu} {\sc X. Hu, L. Zhang}, {\em A new implicit compact difference scheme for the fourth-order fractional diffusion-wave system}, Int. J. Comput. Math., 91 (2014), pp.~2215--2231.

   \bibitem{Kim} {\sc J. U. Kim}, {\em  Exact controllability of an Euler--Bernoulli equation}, SIAM J. Control Optim. 30 (1992), pp.~1001--1023.

   \bibitem{Lasiecka} {\sc I. Lasiecka}, {\em Exponential decay rates for the solutions of Euler--Bernoulli equations with boundary dissipation occurring in the moments only}, J. Differential Equ., 95 (1992), pp.~169--182.

   \bibitem{Lange} {\sc H. Lange, G. Perla Menzala}, {\em Rates of decay of a nonlocal beam equation}, Differ. Integral Equ., 10 (1997), pp.~1075--1092.

    \bibitem{Lele} {\sc S. K. Lele}, {\em Compact finite difference schemes with spectral-like resolution}, J Comput. Phys., 103 (1992), pp.~16--42.

  \bibitem{Li} {\sc S. Li, J. Yu, Z. Liang, G. Zhu}, {\em Stabilization of high eigenfrequencies of a beam equation with generalized viscous damping}, SIAM J. Control Optim., 37 (1999), pp.~1767--1779.

   \bibitem{Liao} {\sc H. Liao, Z. Sun and H. Shi}, {\em Error estimate of fourth-order compact scheme for linear Schr\"odinger equations}, SIAM J. Numer. Anal., 47 (2010), pp.~4381--4401.

   \bibitem{Liu1} {\sc K. Liu, Z. Liu}, {\em Exponential decay of energy of the Euler--Bernoulli beam with locally distributed Kelvin--Voigt damping}, SIAM J. Control Optim., 36 (1998), pp.~1086--1098.

   \bibitem{Liu2} {\sc K. Liu, S. Chen, Z. Liu}, {\em Spectrum and stability for elastic systems with global or local Kelvin-Voigt damping}, SIAM J. Appl. Math., 59 (1998), pp.~651--668.

   \bibitem{LiuZ} {\sc Z. Liu, J. Liu, W. He}, {\em  Boundary control of an Euler--Bernoulli beam with input and output restrictions}, Nonlinear Dynam., 92 (2018), pp.~531--541.

  \bibitem{Lopez} {\sc J. C. Lopez-Marcos},  {\em A difference scheme for a nonlinear partial integrodifferential equation}, SIAM J. Numer. Anal., 27 (1990), pp.~20--31.


  \bibitem{Papanicolaou} {\sc V. G. Papanicolaou}, {\em The periodic Euler--Bernoulli equation}, Trans. Am. Math. Soc., 355 (2003), pp.~3727--3759.

  \bibitem{Sun} {\sc Z. Sun}, {\em An unconditionally stable and $O(\tau^2+h^4)$ order $L^\infty$ convergent difference scheme for linear parabolic equations with variable coefficients}, Numer. Meth. Part. Diff. Eq., 17 (2001), pp.~619--631.


  \bibitem {Wang1} {\sc H. Wang, G. Chen}, {\em Asymptotic locations of eigenfrequencies of Euler--Bernoulli beam with nonhomogeneous structural and viscous damping coefficients}, SIAM J. Control Optim., 29 (1991), pp.~347--367.

   \bibitem {WangJ} {\sc J. M. Wang, G. Q. Xu, S. P. Yung}, {\em Riesz basis property, exponential stability of variable coefficient Euler--Bernoulli beams with indefinite damping}, IMA J. Appl. Math., 70 (2005), pp.~459--477.

    \bibitem{Wang} {\sc Y. Wang} {\em A high-order compact finite difference method and its extrapolation for fractional mobile/immobile convection-diffusion equations}, Calcolo, 54 (2017), pp.~733--768.

   \bibitem {Weeger} {\sc O. Weeger, U. Wever, B. Simeon}, {\em  Isogeometric analysis of nonlinear Euler--Bernoulli beam vibrations}, Nonlinear Dyn., 72 (2013), pp.~813--835.

   \bibitem{Xie} {\sc S. Xie, S. Yi and T. Kwon}, {\em Fourth-order compact difference and alternating direction implicit schemes for telegraph equations},
  Comput. Phys. Commun., 183 (2012), pp.~552--569.

   \bibitem {Xu}  {\sc D. Xu}, {\em Analytical and numerical solutions of a class of nonlinear integro-differential equations with $L^1$ kernels}, Nonlinear Anal. Real. World Appl., 51 (2020), p.~103002.

   \bibitem {Yang} {\sc Z. Yang}, {\em Existence and energy decay of solutions for the Euler--Bernoulli viscoelastic equation with a delay}, Z. Angew. Math. Phys., 66 (2015), pp.~727--745.

    \bibitem{Zhang} {\sc B. Zhang,  H. Fu}, {\em An efficient two-grid high-order compact difference scheme with variable-step BDF2 method for the semilinear parabolic equation}, ESAIM:M2AN, 58 (2024), pp.~421--455.

   \bibitem {Zhao} {\sc Z. Zhao, G. Ren}, {\em A quaternion-based formulation of Euler--Bernoulli beam without singularity}, Nonlinear Dyn., 67 (2012), pp.~1825--1835.









\end{thebibliography}
\end{document}